\documentclass[12pt]{article}
\usepackage[active]{srcltx}
\usepackage{amsmath,amssymb}
\usepackage{amsthm}
\usepackage{latexsym}
\usepackage{color}
\usepackage{dsfont}
\usepackage[normalem]{ulem}
\usepackage{epsfig}
\usepackage{verbatim}
\usepackage{enumerate}
\usepackage{float}
\usepackage{url}
\usepackage{authblk}

\usepackage{epstopdf}

\counterwithin*{equation}{section}
\newtheorem{teo}{Theorem}[section]
\newtheorem{Theorem}{Theorem}[section]
\newtheorem{lema}[teo]{Lemma}

\newtheorem{defi}[teo]{Definition}
\newtheorem{cor}[teo]{Corollary}
\newtheorem{rem}[teo]{Remark}

\begin{document}

\title{\bf Quenching phenomena in a system of non-local diffusion equations}

\author[1,2]{Jos\'e M. Arrieta Algarra\thanks{Email: \texttt{arrieta@mat.ucm.es}}}
\author[1]{Ra\'ul Ferreira\thanks{Email: \texttt{raul\_ferreira@mat.ucm.es}}}
\author[1]{Sergio Junquera\thanks{Corresponding author. Email: \texttt{sejunque@ucm.es}}}

\affil[1]{Departamento de An\'alisis Matem\'atico y Matem\'{a}tica Aplicada, Universidad Complutense de Madrid, 28040 Madrid, Spain.}
\affil[2]{Instituto de Ciencias Matem\'aticas CSIC-UAM-UC3M-UCM, C/ Nicol\'as Cabrera 13-15, Cantoblanco, 28049 Madrid, Spain.}

\date{}

\maketitle

\begin{abstract}
In this paper we study the quenching phenomena occurring in a non-local diffusion system of two equations with intertwined singular absorption terms of the type $u^{-p}$. We prove that there exists a range of multiplicative parameters for which every solution presents quenching, while outside this range there are both global and quenching solutions. We also characterize in terms of the exponents of the absorption terms when the quenching is simultaneous or non-simultaneous and obtain the quenching rates.

\end{abstract}

\begin{section}{Introduction}

In this paper we study the solutions and behaviour of the system
    \begin{equation} \label{1.1}
    \left\{
    \begin{array}{l}
        \displaystyle u_t(x,t)=\int_\Omega J(x-y)u(y,t)\,dy
        +\int_{\mathbb{R}^N\setminus\Omega} J(x-y)\,dy -u(x,t)-\lambda v^{-p}(x,t), \\
        \displaystyle v_t(x,t)=\int_\Omega J(x-y)v(y,t)\,dy
        +\int_{\mathbb{R}^N\setminus\Omega} J(x-y)\,dy-v(x,t)-\mu u^{-q}(x,t),  \\
        u(x,0)=u_0(x)>0; \; v(x,0)=v_0(x)>0,
    \end{array}
    \right.
    \end{equation}
with $x\in\overline{\Omega}$ and $t\in [0,T)$. We consider that $T\in (0,\infty]$ is the maximal existence time of the solution, $\Omega \subset \mathbb{R}^N$ is an open bounded connected smooth domain, the initial data $u_0$ and $v_0$ are positive continuous functions in $\mathcal{C}(\overline{\Omega})$ and the parameters $\lambda,\mu, p, q >0$. The kernel $J:\mathbb{R}^N\to
\mathbb{R}$ is a non-negative $C^1$ function, radially symmetric, decreasing and with $\int_{\mathbb{R}^N} J(s)\,ds=1$.

To simplify the notation, given a solution of \eqref{1.1}, we can consider the following extensions to $\mathbb{R}^N$:
\begin{equation*}
\tilde{u}(x,t)=\left\{
    \begin{array}{ll}
        u(x,t), & x \in{\overline{\Omega}} \\
        1, & x \in \mathbb{R}^N \backslash \overline{\Omega}.
    \end{array}
    \right., \;\;
    \tilde{v}(x,t)=\left\{
    \begin{array}{ll}
        v(x,t), & x \in{\overline{\Omega}} \\
        1, & x \in \mathbb{R}^N \backslash \overline{\Omega}.
    \end{array}
    \right.
\end{equation*}

Then these new functions satisfy the following equations:
\begin{equation} \label{1.2}
    \left\{
    \begin{array}{ll}
        \displaystyle \tilde{u}_t(x,t)= J \ast \tilde{u} (x,t) - \tilde{u} (x,t)-\lambda \tilde{v} ^{-p}(x,t), &x\in{\overline{\Omega}}, \, t\in[0,T) \\
        \displaystyle \tilde{v}_t(x,t)= J \ast \tilde{v} (x,t) - \tilde{v}(x,t)-\mu \tilde{u}^{-q}(x,t), &x\in{\overline{\Omega}}, \, t\in[0,T)  \\
        \tilde{u}(x,t) = \tilde{v}(x,t) = 1, & x \in \mathbb{R}^N \backslash \overline{\Omega}, \, t\in[0,T) \\
        \tilde{u}(x,0)=u_0(x); \; \tilde{v}(x,0)=v_0(x), &x\in{\overline{\Omega}},
    \end{array}
    \right.
\end{equation}

Throughout the paper we will use these two formulations indistinctly, and with some abuse of notation we will write the $\tilde{u}$ as $u$ directly, always considering that we are extending the solution $u(\cdot,t) \in \mathcal{C} (\overline{\Omega})$ by $1$ in $\mathbb{R}^N \backslash \overline{\Omega}$. This is natural in the field of non-local operators, since their nature forces us to set a boundary value for our equation in the whole complement of $\overline{\Omega}$ instead of just at $\partial\Omega$, see \cite{AMRT,BFRW,BV,Fi}. In our case, the convolution in $\mathbb{R}^N$ only makes sense if $u$ and $v$ take values in the whole space, therefore the extension by $1$ is considered as a Dirichlet boundary condition. We also note that the extended function is not continuous at $\partial\Omega$ in general, see \cite{Ch, ChChR}.

Notice that both equations of our system have a singular absorption term. Therefore, the solutions should decrease at some point of $\overline\Omega$ and they could eventually vanish there. If this happens in a finite time $t=T_1$, the corresponding absorption term blows up and the classical solution no longer exists, then $T = T_1$. We say in this situation that quenching happens and $T$ is the quenching time. More precisely we say that a solution of system \eqref{1.1} given by $(u,v)$ presents quenching in finite time $T$ if
$$
\liminf_{t\nearrow T} \min \left\{\min_{\overline\Omega} u(\cdot,t),\min_{\overline\Omega}
v(\cdot,t) \right\} = 0.
$$
We will often refer to a solution that presents quenching as a quenching solution or just that the solution quenches.

Let us observe that there is no reason a priori for both components to present quenching
simultaneously at the quenching time. We say that the quenching is non-simultaneous if only one of the component reaches the zero level at time $T<+\infty$, while it is simultaneous if both components reach the zero level at time $T<+\infty$.

The phenomenon of quenching appears naturally in physical models such as the nonlinear heat conduction in solid hydrogen, see \cite{R}, or the Arrhenius Law in combustion theory, see \cite{CK}.  Quenching  was studied for the first time in \cite{K} for the problem
$$
v_t = v_{xx} + (1-v)^{-1}
$$
where quenching happens when $v$ reaches the value $v=1$. Notice that the absorption term is of the same type as the one we consider with the change of variables $u=1-v$. Since then, the phenomenon of quenching for different problems has been the issue of intensive study for local diffusion operators, see for example the surveys \cite{C,FL,L1,L2,L3} for a single equation and \cite{FPQR, JZZ, PQR, ZW} for systems. In particular in \cite{ZW} the authors consider the local version of \eqref{1.1},
\begin{equation}\label{1.1local}
\left\{
\begin{array}{ll}
u_t=\Delta u-v^{-p}, \ v_t=\Delta v-u^{-q} \qquad & (x,t) \in \overline\Omega \times (0,T)\\
u(x,t)=1=v(x,t) & (x,t) \in\partial\Omega \times (0,T).
\end{array}\right.
\end{equation}
Regarding simultaneous and non-simultaneous quenching, a very clear picture is obtained for radial solutions with $\Omega =B_R(0)$: the quenching must be simultaneous if $p, q \ge 1$, and non-simultaneous when $p < 1 \le q$ or $q < 1 \le p$; if $p, q < 1$ then both simultaneous and non-simultaneous quenching may happen, depending on the initial data. However, nothing is said about the quenching rates.

The problem of quenching with non-local diffusion has been less studied. We cite the papers
\cite{Fe1,ZMZ} where the authors consider the single equation with same non-local diffusion operator as in our system under Dirichlet and Neumann conditions respectively.

The aim of this article is twofold. On the one hand, we are interested in determining whether classical solutions of \eqref{1.1} exist globally in time or they present quenching in finite time. On the other hand, we study the behaviour of the solutions near the maximal existence time. Let us specify first the notion of solution that we use:
\begin{defi}
We say that $(u,v) \in \mathcal{C}^1([0,T), \mathcal{C}(\overline{\Omega}) \times \mathcal{C}(\overline{\Omega}))$ is a classical solution of  system \eqref{1.1} if it satisfies the equations in \eqref{1.1} pointwise for every $(x,t)\in\overline\Omega \times [0,T)$.
\end{defi}
To study the local existence of solutions we will consider a more general $n$-dimensional system
\begin{equation} \label{sistema.general}
(u_i)_t (x,t) = \int_\Omega J_i (x-y) u_i(y,t) dy +f_i(t,x,u_1,\cdots,u_n), \qquad i=1,\cdots,n
\end{equation}
for $(x,t) \in \overline\Omega \times (t_0,t_1)$ under some conditions on the kernels $J_i$ and the functions $f_i$. We will be able to prove existence, uniqueness and regularity of local solutions, see Theorem \ref{prop.existe.general}, and a comparison result, see Lemma \ref{lemacomparacion}. These results are interesting in their own and could be applied to a broad class of problems. In particular, since \eqref{1.1} is a particular case of \eqref{sistema.general}, they will give us the following theorem.

\begin{Theorem} \label{teo.existe}
Let $u_0$ and $v_0$  be two positive functions in $C(\overline{\Omega})$. Then there exists a unique classical solution $(u,v)$ of the
problem \eqref{1.1}, defined in $[0,T)$, where $T$ is the maximal existence time. Moreover, $(u,v)\in C^\infty ([0,T):C(\overline\Omega)\times C(\overline\Omega))$ and, if $T<+\infty$, then
\begin{equation}\label{eq-m}
\lim_{t\nearrow T} \min \left\{\min_{ x\in\overline\Omega} u(x,t),\min_{x\in\overline\Omega}v(x,t) \right\} = 0,
\end{equation}
and we will say that the solution presents quenching.
\end{Theorem}

Our next step is to study the existence of global and quenching solutions, which is closely related with the existence of stationary solutions. In this direction we get the following result.

\begin{Theorem} \label{teo.estacionarias}
There exists an open neighbourhood $U$ of $(0,0)$ in $\mathbb{R}^2$ such that for $(\lambda,\mu)\in ((0,\infty) \times (0,\infty)) \cap \overline{U}$ there exist both global and quenching solutions, whereas for $(\lambda,\mu) \in ((0,\infty) \times (0,\infty)) \cap (\mathbb{R}^2 \backslash \overline{U})$ all solutions present quenching. Moreover, the following are satisfied:

\begin{enumerate} [i)]
    \item if $(\lambda_0,\mu_0) \in ((0,\infty) \times (0,\infty)) \cap \overline{U}$, then $(0,\lambda_0] \times (0,\mu_0] \subset \overline{U}$;
    \item $((0,\infty) \times (0,\infty)) \cap \overline{U} \subset (0,1) \times (0,1)$;
    \item if $(w_1,z_1)$ is a stationary solution of \eqref{1.1} with parameters $(\lambda_1,\mu_1)$, $(w_2,z_2)$ is a stationary solution with parameters $(\lambda_2,\mu_2)$ and $\lambda_1 \leq \lambda_2$, $\mu_1 \leq \mu_2$; then $w_1(x) \geq w_2(x)$ and $z_1(x) \geq z_2(x)$ for every $x\in\overline\Omega$.
\end{enumerate}
\end{Theorem}
In this last result we highlight some of the properties of the set $U$. However, there are still many unanswered questions: it would be interesting to completely determine its geometry with respect to $(\lambda,\mu)$, as well as the behaviour of the stationary solutions when we approach the boundary of $U$ from the interior, where some kind of bifurcation is expected. This is an open problem even in the particular case of the system with only one equation.

Next, we look at whether quenching is simultaneous or non-simultaneous and we obtain a similar result as in the local system \eqref{1.1local} for $\Omega=B_R(0)$. Even more, for $p,q\ge1$ the quenching set of both components coincide. These quenching sets are defined as follows:
\begin{defi}
    Let $(u,v)$ be a solution of \eqref{1.1} that presents quenching at time $T<+\infty$. We define its quenching set as
    \begin{equation*}
        Q((u,v)) = \left\{x\in\overline\Omega : \exists \, t_n \rightarrow T^-, x_n \rightarrow x \text{ such that }  \min\{u(x_n,t_n), v(x_n,t_n) \} \rightarrow 0 \right\}.
    \end{equation*}
    Additionally, we can define the quenching sets associated to each one of the components:
    \begin{equation*}
        Q(u) = \left\{x\in\overline\Omega : \exists \, t_n \rightarrow T^-, x_n \rightarrow x \text{ such that } u(x_n,t_n) \rightarrow 0 \right\}.
    \end{equation*}
    \begin{equation*}
        Q(v) = \left\{x\in\overline\Omega : \exists \, t_n \rightarrow T^-, x_n \rightarrow x \text{ such that } v(x_n,t_n) \rightarrow 0 \right\}.
    \end{equation*}
\end{defi}

\begin{Theorem}\label{teo.simultaneo}
Let $(u,v)$ be a solution of \eqref{1.1} that presents quenching at time $T<+\infty$. Then,

i) if $p,q\geq 1$, quenching is always simultaneous and $Q(u)=Q(v)$;

ii) if $q\ge 1>p$, there exists $\alpha>0$ such that $u(x,t) \geq \alpha$ for every $(x,t)\in\overline\Omega \times [0,T)$;

iii) if $p\ge 1>q$, there exists $\alpha>0$ such that $v(x,t) \geq \alpha$ for every $(x,t)\in\overline\Omega \times [0,T)$;

iv) for $p,q<1$ there can be both simultaneous and non-simultaneous quenching.
\end{Theorem}

We note that the existence of non-simultaneous quenching for a weakly coupled quenching system contrasts with the analogous blow-up problem,
$$
\left\{
\begin{array}{l}
\displaystyle u_t=\int_{\Omega}J(x-y) u(y,t)dy -u(x,t)+v^p(x,t) \\
\displaystyle v_t=\int_{\Omega}J(x-y) v(y,t)dy -v(x,t)+u^q(x,t)
\end{array}\right.
$$
where the blow-up is always simultaneous, see \cite{Fe2}.

Finally we look for the quenching rate. To simplify the notation, we say that
\begin{equation*}
    f(t) \sim g(t) \iff \; \exists \, C_1,C_2>0, t_0 \in [0,T) : \, C_1 f(t) \leq g(t) \leq C_2 f(t) \text{ for } t\in [t_0,T).
\end{equation*}
This essentially gives a similarity in the behaviour of both functions close to time $T$. We start with the non-simultaneous case.

\begin{Theorem}\label{teo.tasas.nosimultanea}
Let $(u,v)$ be a solution of \eqref{1.1} that presents quenching at time $T<+\infty$.

i) Assume that  $v(x,t) \geq \delta > 0$ for every $(x,t)\in\overline\Omega \times [0,T)$. Then $q<1$ and
$$
\min_{x\in\overline\Omega} u(x,t)\sim (T-t).
$$

ii) Assume that  $u(x,t)\geq \delta >0$ for every $(x,t)\in\overline\Omega \times [0,T)$. Then $p<1$ and
$$
\min_{x\in\overline\Omega} v(x,t)\sim (T-t).
$$
\end{Theorem}

The simultaneous case is more involved because we need a relation between the minimum of both components.
Let us define $x_u(t)$ and $x_v(t)$ as follows
$$
u(x_u(t),t):= \min_{\overline\Omega} u(\cdot,t),\qquad
v(x_v(t),t):= \min_{\overline\Omega} v(\cdot,t)
$$
Then to obtain the quenching rates, we need to relate $u(x_u(t),t)$ with $v(x_u(t),t)$ and $u(x_v(t),t)$ with $v(x_v(t),t)$. This is given in Lemma  \ref{lema-estimaciones} for $\max\{p,q\}\ge1$ while for $p,q<1$ we must impose that both components reach the minimum at the same point for all time, $x_u(t)=x_v(t)$. This is the case if $\Omega$ is a ball and both components are radially increasing.

\begin{Theorem} \label{lema.qrate.sim}
Let $\max\{p,q\}\ge1$ and $(u,v)$ be a solution of \eqref{1.1} that presents quenching at time $T<+\infty$. Then

$i)$ for $p,q>1$,
$$
u(x_u (t),t)\sim (T-t)^{\frac{p-1}{pq-1}}, \qquad v(x_u (t),t)\sim (T-t)^{\frac{q-1}{pq-1}}
$$
and
$$
u(x_v (t),t)\sim (T-t)^{\frac{p-1}{pq-1}}, \qquad v(x_v (t),t)\sim (T-t)^{\frac{q-1}{pq-1}}.
$$

$ii)$ for $p>1=q$,
$$
u(x_u (t),t)\sim (T-t) |\log (T-t)|^{\frac{-p}{1-p}}, \qquad v(x_u (t),t)\sim |\log (T-t)|^{\frac{1}{1-p}}
$$
and
$$
u(x_v (t),t)\sim (T-t) |\log (T-t)|^{\frac{-p}{1-p}}, \qquad v(x_v (t),t)\sim |\log (T-t)|^{\frac{1}{1-p}}.
$$

$iii)$ for $p=1<q$,
$$
u(x_u (t),t)\sim |\log (T-t)|^{\frac{1}{1-q}}, \qquad v(x_u (t),t)\sim (T-t) |\log (T-t)|^{\frac{-q}{1-q}}
$$
and
$$
u(x_v (t),t)\sim |\log (T-t)|^{\frac{1}{1-q}}, \qquad v(x_v (t),t)\sim (T-t) |\log (T-t)|^{\frac{-q}{1-q}}.
$$

$iv)$ for $p=1=q$,
$$
u(x_u (t),t)\sim (T-t)^{\frac{\lambda}{\lambda+\mu}}, \qquad v(x_u (t),t)\sim (T-t)^{\frac{\mu}{\lambda+\mu}}
$$
and
$$
u(x_v (t),t)\sim (T-t)^{\frac{\lambda}{\lambda+\mu}}, \qquad v(x_v (t),t)\sim (T-t)^{\frac{\mu}{\lambda+\mu}}.
$$
\end{Theorem}

\begin{Theorem} \label{lema.qrate.sim.menor1}
    Let $p,q< 1$ and $(u,v)$ be a solution of \eqref{1.1} that presents quenching at time $T<+\infty$ such that there exists some $t_0\in[0,T)$ for which $x_u(t) = x_v(t) = x(t)$ for every $t\in[t_0,T)$. If the quenching is simultaneous then
      $$
      u(x(t),t)\sim (T-t)^{\frac{p-1}{pq-1}}, \qquad
      v(x(t),t)\sim (T-t)^{\frac{q-1}{pq-1}}.
      $$
\end{Theorem}

We describe now the contents of the paper. In Section \ref{sec-exist}, we prove Theorem \ref{teo.existe}. As a matter of fact, we will obtain an existence, uniqueness and regularity result of solutions for a problem which is more general than \eqref{1.1}, see Theorem \ref{prop.existe.general}. We also obtain comparison results for this more general class of problems, see Lemma \ref{lemacomparacion}. These two results are interesting by themselves, and they will be applicable in other more general situations.

In Section 3 we prove Theorem \ref{teo.estacionarias}, that is, we will obtain necessary and sufficient conditions for $(\lambda,\mu)$ under which stationary solutions exist for system \eqref{1.1}. We also obtain some auxiliary results on the properties of the stationary solutions of this problem, see Lemma \ref{condicionquenching} and Lemma \ref{lema.globalesmenorque1}.

In Section 4 we prove Theorem \ref{teo.simultaneo}. We will do it first for the case $\max\{p,q\} \geq 1$, in which a functional inequality will give us all the information. Then we will treat the case $p,q\leq 1$, which is more involved and will be solved with the help of a shooting argument.

In Section 5 we first prove Theorem \ref{teo.tasas.nosimultanea}, which gives us the quenching rate for non-simultaneous quenching. Then we prove Theorem \ref{lema.qrate.sim} and Theorem \ref{lema.qrate.sim.menor1}, which consider the quenching rate for simultaneous quenching and are more difficult to prove.

Finally, in Section 6 we provide some numerical simulations that complement the results obtained in the paper.

\end{section}

\begin{section}{Existence, uniqueness, regularity of solutions and comparison results}
\label{sec-exist}

In this section, we will provide a proof of Theorem \ref{teo.existe}, that is, we will prove that \eqref{1.1} admits maximal classical solutions, and they are unique and smooth in the time variable. As a matter of fact, we will consider a more general system of equations and prove the existence of maximal classical solutions for it first.

Let $n\in\mathbb{N}$ and assume the following:
\begin{itemize}
    \item $J_i \in \mathcal{C}(\mathbb{R}^N)$ are non-negative with $\int_{\mathbb{R}^N}J_i (s)ds = C_i < \infty$ for every $i=1,\dots,n$.
    \item $\mathcal{I} = I_1 \times \dots I_n \subset \mathbb{R}^n$, where $I_i = (a_i,b_i)$ with $a_i,b_i \in \mathbb{R} \cup \{ \pm \infty \}$, $a_i<b_i$ for every $i=1,\dots,n$.
    \item $f = (f_i)_{i=1}^n \in \mathcal{C}(\overline\Omega \times (t_0,t_1) \times \mathcal{I}, \mathbb{R}^n)$, with $t_0,t_1 \in \mathbb{R} \cup \{\pm \infty \}$ such that $t_0<t_1$, and $f(x,t,s)$ is locally Lipschitz in the variable $s\in\mathcal{I}$ uniformly with respect to $(x,t) \in \overline\Omega \times (t_0,t_1)$.
    \item $u_0 = (u_0^i)_{i=1}^n \in \mathcal{C}(\overline\Omega, \mathbb{R}^n)$ such that $u_0^i(x) \in I_i$ for every $i=1,\dots,n$ and $x\in\overline\Omega$.
    \item $\tau_0 \in (t_0,t_1)$.
\end{itemize}

Define
\begin{equation*}
    J(x) = \begin{pmatrix}
            J_1 (x) & 0 & 0 & \dots & 0 \\
            0 & J_2(x) & 0 & \dots & 0 \\
            \vdots & \vdots & \ddots & \vdots & \vdots \\
            0 & \dots & 0 & J_{n-1} (x) & 0 \\
            0 & \dots & 0 & 0 & J_n (x)
        \end{pmatrix}, \; \;
    u(x,t) = \begin{pmatrix}
        u_1(x,t) \\
        u_2(x,t) \\
        \vdots \\
        u_{n-1}(x,t) \\
        u_n (x,t)
    \end{pmatrix},
\end{equation*}
and consider the following system of equations:
\begin{equation}\label{auxsistema.general}
    \left\{
        \begin{array}{l}
             \displaystyle u_t(x,t) = \int_\Omega J(x-y) u(y,t) dy + f(x,t,u(x,t)),  \\
             u(x,\tau_0) = u_0 (x),
        \end{array}
    \right.
\end{equation}
where $(x,t)\in\overline{\Omega} \times (t_0,t_1)$.

We say that a function $u\in \mathcal{C}^1((\tau,\tau^\prime), \mathcal{C}(\overline\Omega,\mathbb{R}^n))$ with $(\tau, \tau^\prime) \subset (t_0,t_1)$ and $\tau_0 \in (\tau, \tau^\prime)$ is a classical solution of this system if it satisfies the equations in \eqref{auxsistema.general} pointwise for every $(x,t)\in\overline\Omega\times(\tau,\tau^\prime)$.

\begin{teo} \label{prop.existe.general}
Under the assumptions and notation above, there exists a unique maximal classical solution $u$ of system \eqref{auxsistema.general}, defined in $(T_0,T_1) \subset (t_0,t_1)$ with $T_0<\tau_0<T_1$. Moreover, if $T_1 < t_1$, we have either
$$\limsup_{t\to T_1^-} \left( \sup_{x\in\overline\Omega} \|u(x,t)\| \right) = +\infty$$
or
$$\lim_{t\to T_1^-} \left( \inf_{x\in\overline\Omega} dist(u(x,t), \partial \mathcal{I}) \right) = 0,$$
where $\|s\| = \max_{i=1,\dots,n} |s_i|$ if $s = (s_i)_{i=1}^n \in \mathbb{R}^n$. Similarly, if $T_0>t_0$, we have the same alternative above with $t\to T_0^+$.

Moreover, if $f$ is $\mathcal{C}^k$ in $t$ and $s$, then $u \in \mathcal{C}^{k+1}((T_0,T_1), \mathcal{C}(\overline\Omega,\mathbb{R}^n))$.
\end{teo}

\begin{proof}[Proof]
We will start by proving existence and uniqueness of local solutions, that is, solutions defined in a short interval of time $[\tau_0-h,\tau_0+h]$. To accomplish this, we will use the Banach fixed point Theorem. We know that, for every $i=1,2,\dots,n$, $u_0^i$ is a continuous function in $\overline\Omega$, which is a compact subset of $\mathbb{R}^N$. Therefore, there exist some $\overline{a_i},\overline{b_i} \in \mathbb{R}$ with $a_i<\overline{a_i}<\overline{b_i}<b_i$, such that $u_0^i(x) \in [\overline{a_i}, \overline{b_i}]$ for every $x\in\overline{\Omega}$. Then take $0 < \delta_i < \frac{1}{2} \min\{|a_i - \overline{a_i}|,|b_i-\overline{b_i}|\}$ and $\delta = \min_i \{\delta_i\}_{i=1}^n$. Define the compact subset $M_\delta = \Pi_{i=1}^n [\overline{a_i}-\delta,\overline{b_i}+\delta] \subset \mathcal{I}$. Next, take $R$ such that $[\tau_0 - R ,\tau_0 + R] \subset (t_0,t_1)$ and, since $f(x,t,s)$ is locally Lipschitz in $s \in \mathcal{I}$, take $L_\delta$ as the Lipschitz constant for $f$ in the closed bounded set $\overline\Omega \times [\tau_0-R, \tau_0+R] \times M_\delta$, that is,
\begin{equation*}
    \|f(x,t,s_1) - f(x,t,s_2)\| \leq L_\delta \|s_1 - s_2\|,
\end{equation*}
for every $(x,t)\in\overline\Omega \times [\tau_0-R, \tau_0+R]$ and $s_1, s_2 \in M_\delta$. Furthermore, define the following quantities:
\begin{equation*}
    \alpha = \max_{i=1,\dots,n} \{|\overline{a_i}-\delta|, |\overline{b_i}+\delta| \}, \;\;\; \beta = \sup_{(x,t,s)\in \overline\Omega \times [\tau_0-R, \tau_0+R] \times M_\delta} \|f(x,t,s)\|, \;\;\; C = \max_{i=1,\dots,n} C_i.
\end{equation*}
Next, take $h>0$ such that it satisfies
\begin{equation*}
    h < \min \left\{ R,  \frac{\delta}{\alpha C + \beta}, \frac{1}{C+L_\delta} \right\}
\end{equation*}

and consider the following closed subset of the Banach space $\mathcal{C}(\overline\Omega \times [\tau_0 - h, \tau_0 + h], \mathbb{R}^n)$ endowed with the usual norm:
\begin{equation*}
    X = \{ \varphi \in  \mathcal{C}(\overline\Omega \times [\tau_0 - h, \tau_0 + h], \mathbb{R}^n) : \|\varphi(\cdot,t) - u_0 \|_\infty \leq \delta, \, \forall t\in [\tau_0 - h, \tau_0 + h] \},
\end{equation*}
where $\|v\|_\infty = \sup_{\overline\Omega} \|v (x)\|$ for every $v\in \mathcal{C}(\overline\Omega, \mathbb{R}^n)$.

Notice that $h$ only depends on the integral of $J$, on some bounds of $f$ and, more importantly, on $R>0$ and $\delta>0$. $R$ depends on the distance of $\tau_0$ to $t_0$ and $t_1$, and $\delta$ arises from the distance of the initial data to the boundary of the set $\mathcal{I}$. This is a fact that will come into play later.

We introduce the nonlinear operator $\Phi: X \longrightarrow \mathcal{C}(\overline\Omega \times [\tau_0-h, \tau_0+h], \mathbb{R}^n)$ defined as
\begin{equation*}
    \Phi(u) (x,t) = u_0 (x) +  \int_{\tau_0}^t \left(  \int_\Omega J (x-y) u(y,s)\,dy \right) ds + \int_{\tau_0}^t f(x,s,u(x,s)) ds = u_0 + \theta_1 + \theta_2,
\end{equation*}
where $x\in\overline\Omega$ and $t\in[\tau_0-h,\tau_0+h]$.

First, we show that $\Phi$ maps $X$ into $X$. Notice that, if $u\in X$, $u(x,t) \in M_\delta \subset \mathcal{I}$ for every $(x,t) \in \overline\Omega \times [\tau_0 - h, \tau_0 + h]$. Then $\theta_2$ is well-defined and continuous in $\overline\Omega \times [\tau_0 - h, \tau_0 + h]$. Moreover, as $u(x,t)$ is bounded and $\int_{\mathbb{R}^N} J_i(s) ds = C_i$ for each $i=1,2,\dots,n$, then $\theta_1$ is also well-defined and continuous in $\overline\Omega \times [\tau_0 - h, \tau_0 + h]$.

Next, thanks to the definition of $h$,
\begin{equation*}
\begin{array}{ll}
    \|\Phi (u) (x,t) - u_0(x) \| &\leq  \Big\| \displaystyle\int_{\tau_0}^t \left(  \int_\Omega J (x -y) u(y,s)\,dy \right) ds \Big\| + \Big\| \int_{\tau_0}^t f(x,s,u(x,s)) ds \Big\| \\
    &\leq \alpha \cdot h \cdot \displaystyle\max_{i=1,\dots,n} C_i + \beta \cdot h = h(\alpha C + \beta) \leq \delta,
    \end{array}
\end{equation*}
for every $x\in\overline\Omega$ and $t\in[\tau_0-h,\tau_0+h]$. Therefore, $\Phi$ maps $X$ into $X$.

On the other hand, we consider $u,v \in X$ and $(x,t) \in \overline\Omega \times [\tau_0-h, \tau_0+h]$ and study the quantity
\begin{equation*}
\begin{array}{ll}
    \| \Phi (u)(x,t) - \Phi(v)(x,t) \| &\leq \Big\| \displaystyle\int_{\tau_0}^t \left(  \int_\Omega J (x -y) (u(y,s)-v(y,s))\,dy \right) ds \Big\| \\
    &+ \Big\| \displaystyle\int_{\tau_0}^t (f(x,s,u(x,s))-f(x,s,v(x,s))) ds \Big\| \\
    & \displaystyle \leq (C \cdot h + L_\delta \cdot h) \sup_{t \in [\tau_0-h,\tau_0+h]} \|u(\cdot,t)-v(\cdot,t)\|_\infty.
\end{array}
\end{equation*}
Therefore we find that
\begin{equation*}
    \sup_{t \in [\tau_0-h,\tau_0+h]} \| \Phi (u)(\cdot,t) - \Phi(v)(\cdot,t) \|_\infty \leq h(C+L_\delta) \sup_{t \in [\tau_0-h,\tau_0+h]} \|u(\cdot,t)-v(\cdot,t)\|_\infty,
\end{equation*}
which is a strict contraction because $h(C+L_\delta)<1$ thanks to the definition of $h$.

Then, by Banach's fixed point theorem, there exists a unique $w \in X$ such that $w = \Phi(w)$. Hence we have
\begin{equation} \label{eq.puntofijo}
    w(x,t) = u_0 (x) +  \int_{\tau_0}^t \left(  \int_\Omega J (x-y) w(y,s)\,dy \right) ds + \int_{\tau_0}^t f(x,s,w(x,s)) ds,
\end{equation}
for every $(x,t) \in \overline\Omega \times [\tau_0-h,\tau_0+h]$.

We notice that the integrands in the right hand side of \eqref{eq.puntofijo} are continuous in time because $w\in\mathcal{C}(\overline\Omega \times [\tau_0-h, \tau_0+h], \mathbb{R}^n)$ and $w(x,t) \in M_\delta \subset \mathcal{I}$, for every $(x,t)\in\overline\Omega \times [\tau_0-h,\tau_0+h]$. Then by the Fundamental Theorem of Calculus, every term in the right hand side is $\mathcal{C}^1$ in time. This means that $w\in \mathcal{C}^1([\tau_0-h, \tau_0+h], \mathcal{C}(\overline\Omega,\mathbb{R}^n))$. Differentiating expression \eqref{eq.puntofijo} in time, we see that $w$ fulfills the first equation of \eqref{auxsistema.general} pointwise. Moreover, $w$ fulfills $w(x,\tau_0) = u_0(x)$ for every $x\in\overline\Omega$ from \eqref{eq.puntofijo}, so we conclude that $w$ is a classical local solution of \eqref{auxsistema.general}. 

To obtain now solutions defined in a maximal interval of time $(T_0,T_1)$, we proceed with a standard argument. As long as $w(\cdot, \tau_0 + h) \in \mathcal{I}$ (or $w(\cdot, \tau_0 - h) \in \mathcal{I}$), we will be able to repeat the previous procedure, choosing appropriate $\delta^\prime$ and $h^\prime$, and we can extend the solution to $[\tau_0+h, \tau_0+h+h^\prime]$. Therefore we can obtain a solution $w$ defined in a maximal interval of existence, $(T_0,T_1)$. Let us show that this interval satisfies that, if $T_1<t_1$, either $\limsup_{t\nearrow T_1} \sup_{x\in\overline\Omega} \|w(x,t)\| = +\infty$ or $\lim_{t\nearrow T_1} \inf_{x\in\overline\Omega} dist(w(x,t), \partial \mathcal{I}) = 0$. Indeed, suppose that neither happens. That means that, on one hand, there is a $\overline{t}<T_1$ and $0<D<\infty$ such that $\sup_{x\in\overline\Omega} \|w(x,t)\| \leq D$ for every $t\in[\overline{t},T_1)$. On the other hand, there exists a sequence of times $\{s_n\}_n$ that satisfies $s_n \xrightarrow{n} T_1^-$ and a real number $a>0$ such that $\inf_{x\in\overline\Omega} dist(w(x,s_n), \partial \mathcal{I}) \geq a$ for every $n\in\mathbb{N}$. Take $n_0\in\mathbb{N}$ such that $s_n \geq \overline{t}$ for every $n\geq n_0$. Then $w(\cdot, s_n) \in \mathcal{I}$ for every $n\geq n_0$, and we can extend the solution for these times according to the previous procedure. Each $R_n$ depends on the distance of $s_n$ to $t_0$ and $t_1$. Since $T_1$ satisfies $t_0<\overline{t}<T_1<t_1$ and $s_n \in [\overline{t}, T_1)$ for every $n\geq n_0$, we can take $\overline{R} < \min\{t_1-T_1, \overline{t}-t_0\}$ and choose $R_n =\overline{R}$ for every $n\geq n_0$. Each $\delta_n$ depends on the distance of the initial data, in this case $w(\cdot, s_n)$, to the boundary of $\mathcal{I}$. In this case, this distance is bounded by $a$. Then we can take $\delta_a >0$ such that $\delta_n = \delta_a$ for every $n \geq n_0$. Since each $h_n$ only depends on $\overline{R}$ and $\delta_a$,  we can also take $h_a>0$ such that $h_n=h_a$ for every $n \geq n_0$, and then extend the solution to intervals $[s_n, s_n + h_a]$ for every $n\geq n_0$. Taking $n$ large enough, this would mean that our solution $w$ is defined in an interval $(T_0,T_1 + \frac{h_a}{2})$, which is a contradiction. The proof in the case $T_0>t_0$ is analogous.

Finally, we take $u$ as a solution of \eqref{auxsistema.general} and observe that, as $u(x,t) \in \mathcal{I}$ in $\overline\Omega \times (T_0, T_1)$, all the terms in the right hand side of the first equation are continuous in $\overline\Omega \times (T_0,T_1)$. Then $u_t$ is also continuous in $\overline\Omega \times (T_0,T_1)$ and therefore $u \in \mathcal{C}^1((T_0,T_1), \mathcal{C}(\overline\Omega,\mathbb{R}^n))$. If $f$ is $\mathcal{C}^k$ in $t$ and $s$, iterating this same argument we get that $u \in \mathcal{C}^{k+1}((T_0,T_1), \mathcal{C}(\overline\Omega,\mathbb{R}^n))$.
\end{proof}

Next, we will prove comparison results for system \eqref{auxsistema.general}. For this, we need some extra conditions on the nonlinearity $f$.

Assume that $f$ can be written as $f_i (x,t,s) = h_i (x,t) s_i + g_i (x,t,s)$  for every $(x,t,s) \in \overline\Omega \times (t_0,t_1) \times \mathcal{I}$ and $i=1,\dots,n$, with $h_i \in \mathcal{C}(\overline\Omega \times (t_0,t_1))$ and $g_i \in \mathcal{C}(\overline\Omega \times (t_0,t_1) \times \mathcal{I})$. Define

\begin{equation*}
    h(x,t) = \begin{pmatrix}
            h_1 (x,t) & 0 & 0 & \dots & 0 \\
            0 & h_2(x,t) & 0 & \dots & 0 \\
            \vdots & \vdots & \ddots & \vdots & \vdots \\
            0 & \dots & 0 & h_{n-1} (x,t) & 0 \\
            0 & \dots & 0 & 0 & h_n (x,t)
        \end{pmatrix}, \; \;
    g(x,t,s) = \begin{pmatrix}
        g_1(x,t,s) \\
        g_2(x,t,s) \\
        \vdots \\
        g_{n-1}(x,t,s) \\
        g_n (x,t,s)
    \end{pmatrix}.
\end{equation*}

This way it is clear that $f(x,t,s) = h(x,t)s + g(x,t,s)$, and so system \eqref{auxsistema.general} transforms into
\begin{equation}\label{auxsistema.comparacion}
    \left\{
        \begin{array}{l}
             \displaystyle u_t(x,t) = \int_\Omega J(x-y) u(y,t) dy + h(x,t)u(x,t) + g(x,t,u(x,t)),  \\
             u(x,\tau_0) = u_0 (x),
        \end{array}
    \right.
\end{equation}
with $(u_0) = (u_0^i)_{i=1}^n \in \mathcal{C}(\overline\Omega, \mathbb{R}^n)$.

\begin{defi}  \label{defsupersolucion}
    We say that, given $\tau_0 \in (t_0,t_1)$ and $\tau \in (\tau_0,t_1]$, a function $\overline{u} = (\overline{u}^i)_{i=1}^n \in \mathcal{C}^1([\tau_0,\tau), \mathcal{C}(\overline{\Omega}, \mathbb{R}^n))$ is a supersolution of \eqref{auxsistema.comparacion} in $[\tau_0,\tau)$ if $\overline{u}(x,t)\in \mathcal{I}$ for every $(x,t)\in\overline\Omega \times [\tau_0,\tau)$, and it satisfies
    \begin{equation*}
    \left\{
        \begin{array}{lr}
             \displaystyle \overline{u}^i_t(x,t) \geq \int_\Omega J_i(x-y) \overline{u}^i (y,t) dy + h_i(x,t)\overline{u}^i (x,t) + g_i(x,t,\overline{u}(x,t)), \;\;\; &i=1,\dots,n, \\
             \overline{u}^i(x,\tau_0) \geq u_0^i (x), &i=1,\dots,n,
        \end{array}
    \right.
    \end{equation*}
    for every $(x,t)\in\overline\Omega \times [\tau_0,\tau)$ and $i=1,\dots,n$. Conversely, $\underline{u} = (\underline{u}^i)_{i=1}^n \in \mathcal{C}^1([\tau_0,\tau), \mathcal{C}(\overline{\Omega}, \mathbb{R}^n))$ is a subsolution of \eqref{auxsistema.comparacion} in $[\tau_0,\tau)$ if $\underline{u}(x,t)\in \mathcal{I}$ and $\underline{u}$ fulfills the reverse inequalities for every $(x,t)\in\overline\Omega \times [\tau_0,\tau)$.
\end{defi}

We can prove the following comparison result for system \eqref{auxsistema.comparacion}:

\begin{lema} \label{lemacomparacion}
    Let $\overline{u}$ be a supersolution of \eqref{auxsistema.comparacion} in $[\tau_0,\overline{\tau})$ and $\underline{u}$ a subsolution of \eqref{auxsistema.comparacion} in $[\tau_0,\underline{\tau})$, and take $\tau = \min\{\overline\tau, \underline\tau\}$. If we assume that $g$ is $\mathcal{C}^1$ with respect to $s$ and
    \begin{equation} \label{kamke.condition}
        \frac{\partial g_i}{\partial u^j} (x,t,s) \geq 0, \;\;\; \text{for every } \; i \neq j \text{ and } (x,t,s) \in \overline\Omega \times [\tau_0,\tau) \times \mathcal{I},
    \end{equation}
    then $\overline{u}^i (x,t) \geq \underline{u}^i (x,t)$ for every $i=1,\dots, n$, $(x,t)\in\overline\Omega \times [\tau_0,\tau)$.
\end{lema}

\begin{proof}
   First, we choose $\tau_1 \in [\tau_0, \tau)$ and take
   \begin{equation*}
       m =\min_{i=1,\dots,n} \left\{\min_{(x,t) \in \overline\Omega \times [\tau_0,\tau_1]} \underline{u}^i(x,t) \right\}, \;\;\; M = \max_{i=1,\dots,n} \left\{\max_{(x,t) \in \overline\Omega \times [\tau_0,\tau_1]} \overline{u}^i(x,t) \right\}.
   \end{equation*}
   This is possible because $\underline{u}^i$ and $\overline{u}^i$ are continuous functions and $\overline\Omega \times [\tau_0,\tau_1]$ is a compact set. Then take $\gamma>0$ such that
   \begin{equation} \label{definiciongamma}
       \gamma > \max_{i=1,\dots,n} \left\{ \int_\Omega J_i(x-y) dy - h_i(x,t) + \sum_{j=1}^n \frac{\partial g_i}{\partial u^j}(x,t,s)\right\}
   \end{equation}
   for every $(x,t,s) \in \overline\Omega \times [\tau_0,\tau_1] \times [m,M]^n$. This is possible because $J_i$, $h_i$, $\frac{\partial g_i}{\partial u^j}$ are continuous functions and $\overline\Omega \times [\tau_0,\tau_1] \times [m,M]^n$ is a compact set so they are bounded in it. Let us define, for every $i=1,\dots,n$, $x\in\overline\Omega$ and $t\in[\tau_0,\tau_1]$,
$$
z^i (x,t)=\overline u^i  (x,t)-\underline u^i (x,t)+\varepsilon e^{\gamma t}.
$$
Notice that $z^i (x,\tau_0)>0$ for every $i=1,\dots,n$ and $x\in\overline\Omega$. Therefore, by continuity, $z^i (x,t)>0$ in $\overline\Omega \times [\tau_0,\tau_0+\delta)$ for some small $\delta\in [0,\tau_1-\tau_0]$. Let us show that $z^i(x,t)>0$ in $\overline\Omega \times [\tau_0,\tau_1]$ for every $i=1,\dots,n$. In other case, there exists a first time $\overline{t} \in (\tau_0,\tau_1]$ such that $z^i(x,t)>0$ in $(x,t) \in \overline\Omega \times [\tau_0,\overline{t})$ for every $i=1,\dots,n$, and $\min_{i=1,\dots,n} \{ \min_{x\in\overline{\Omega}} z^i (x,\overline{t})\} = 0$. There exists $\overline{x} \in \overline\Omega$ and $k \in \{1,\dots,n\}$ such that $z^k(\overline{x},\overline{t})=0$

From the inequalities satisfied by $\overline{u}_t$ and $\underline{u}_t$, it is easily seen that, for every $x\in\overline\Omega$ and $t\in[\tau_0,\tau_1]$,
\begin{align*}
    z^k_t(x,t) \geq& \int_\Omega J_k (x-y) (z^k(y,t)-\varepsilon e^{\gamma t}) dy - h_k(x,t) (z^k(x,t) -\varepsilon e^{\gamma t}) \\
    &+g_k (x,t, \overline{u}) - g_k (x,t, \underline{u}) + \varepsilon \gamma e^{\gamma t},
\end{align*}
where we have used that $\overline{u}^k (x,t)-\underline{u}^k (x,t) = z^k (x,t) - \varepsilon e^{\gamma t}$. Reordering terms and applying the Mean Value Theorem on the last argument of $g_k$ we get, in $\overline\Omega \times [\tau_0,\tau_1]$,
\begin{align*}
    z^k_t(x,t) \geq& \int_\Omega J_k (x-y) z^k(y,t) dy - h_k(x,t)z^k(x,t) + \sum_{j=1}^n \frac{\partial g_k}{\partial u^j}(x,t,\xi) z^j (x,t) \\
    &  + \varepsilon e^{\gamma t} \left(\gamma - \int_\Omega J_k (x-y) dy + h_k(x,t) - \sum_{j=1}^n \frac{\partial g_k}{\partial u^j}(x,t,\xi) \right),
\end{align*}
where $\xi = (\xi(x,t))_{i=1}^n$ and $\overline{u}^i (x,t)\geq \xi^i(x,t) \geq \underline{u}^i(x,t)$, for every $i=1,\dots,n$, $x\in\overline\Omega$ and $t\in[\tau_0,\tau_1]$. Note that $\xi(x,t) \in [m,M]^n \subset \mathcal{I}$ for every $(x,t)\in \overline\Omega \times [\tau_0,\tau_1]$ because $\overline{u}(x,t), \underline{u}(x,t) \in [m,M]^n$ for every $(x,t)\in \overline\Omega \times [\tau_0,\tau_1]$.

Finally, we evaluate the expression in $(\overline{x},\overline{t})$ to use the fact that $z^k (\overline{x},\overline{t})=0$ and condition \eqref{definiciongamma} to obtain:
\begin{equation*}
    z^k_t(\overline{x}, \overline{t}) > \int_\Omega J_k (\overline{x}-y) z^k (y,\overline{t}) dy + \sum_{j=1}^{k-1} \frac{\partial g_k}{\partial u^j}(x,t,\xi) z^j (\overline{x},\overline{t}) + \sum_{j=k+1}^n \frac{\partial g_k}{\partial u^j}(x,t,\xi) z^j (\overline{x},\overline{t}) \geq 0,
\end{equation*}
where we have used the positivity of $J_k$ and $z^k$ and \eqref{kamke.condition} to get the last inequality. However, since $\overline{t} \in (\tau_0,\tau_1]$ is the first time in which $z^k$ vanishes, it is clear that $z^k_t(\overline{x},\overline{t}) \leq 0$, so we have reached a contradiction. Therefore, $z^i(x,t) >0$ for every $i=1,\dots,n$; $x\in\overline{\Omega}$ and $t\in[\tau_0,\tau_1]$. Taking the limit $\varepsilon\rightarrow 0$, we obtain that $\overline{u}^i(x,t) \geq \underline{u}^i(x,t)$ for every $i=1,\dots,n$ in $\overline\Omega \times [\tau_0,\tau_1]$. Since this is true for every $\tau_1 \in [\tau_0,\tau)$, the result is proven.
\end{proof}

\begin{rem} \label{rem.kamke}
    Condition \eqref{kamke.condition} is sometimes known as the Kamke condition, as seen in \cite{S}.
\end{rem}

We are now in a position to provide a proof of Theorem \ref{teo.existe}.

\textit{Proof of Theorem \ref{teo.existe}}. Let us observe that system \eqref{1.1} is a particular case of \eqref{auxsistema.comparacion} with $n=2$, $I_1 = I_2 = (0,\infty)$, $J_1(x) = J_2(x)= J(x)$, $h_1(x,t) = h_2(x,t) = -1$, $g_1(x,t,u,v) = \int_{\mathbb{R}^N \backslash \Omega} J(x-y) dy -\lambda v^{-p}$, and $g_2(x,t,u,v)= \int_{\mathbb{R}^N \backslash \Omega} J(x-y) dy -\mu u^{-q}$ for every $x\in\overline\Omega$, $t\in[\tau_0,\tau)$ and $u,v\in(0,\infty)$. Also, since $g_1,g_2$ are $\mathcal{C}^\infty$ in $u,v \in (0,\infty)$, then from Theorem \ref{prop.existe.general} the solution is also $\mathcal{C}^\infty$.

Hence, if $u_0, v_0$ are strictly positive and bounded continuous functions we are in the hypotheses of Theorem \ref{prop.existe.general} and Lemma \ref{lemacomparacion}. Then, there exists a local in time classical solution and our comparison principle holds. Moreover, if $T<\infty$ we have either
$$
\limsup_{t\to T}\left( \max\left\{\max_{x\in\overline\Omega} u(x,t), \max_{x\in\overline\Omega} v(x,t)\right\}\right)  = +\infty
$$ 
or
$$
\liminf_{t\to T} \left( \min\left\{\min_{x\in\overline\Omega} u(x,t), \min_{x\in\overline\Omega} v(x,t)\right\}\right)  = 0.
$$
Observe that in this case, $\inf_{x\in\overline\Omega} dist((u(x,t),v(x,t)), 0) = \min \{ \min_{x\in\overline\Omega} u(x,t), \min_{x\in\overline\Omega} v(x,t) \}$ because the components $u$ and $v$ are strictly positive and $I_1 = I_2 = (0,\infty)$.

Therefore, to prove Theorem \ref{teo.existe} we only have to see that if $T<\infty$ then the solution is bounded, which implies that the second option is satisfied. This result is given by comparison in the following lemma, which finishes the proof. \qed

\begin{lema} \label{lema.MNsuper}
    Let $u_0$ and $v_0$  be two positive functions in $C(\overline{\Omega})$ and $(u,v) \in \mathcal{C}^1([0,\tau), \mathcal{C}(\overline{\Omega})\times \mathcal{C}(\overline{\Omega}))$ the corresponding solution of system \eqref{1.1}. Define the quantities
    \begin{equation*}
        M=\max\{1,\|u_0\|_\infty\}, \;\;\; N=\max\{1,\|v_0\|_\infty\}.
    \end{equation*}
    Then $u(x,t) \leq M$ and $v(x,t) \leq N$ for every $(x,t)\in \overline\Omega \times[0,\tau)$.
\end{lema}
\begin{proof}
    We prove that $(M,N)$ is a supersolution of system \eqref{1.1} with initial data $(u_0,v_0)$. On one hand, we see clearly that $M\geq u_0(x)$ and $N \geq v_0(x)$ for every $x\in\overline\Omega$. On the other hand, since $M, N\geq 1$, we see that
    \begin{equation*}
        M \int_{\Omega} J(x-y) dy + \int_{\mathbb{R}^N \backslash \Omega} J(x-y) dy - M - \lambda N^{-p} \leq M - M - \lambda N^{-p} \leq 0.
    \end{equation*}
    The other inequality is analogous. Therefore, thanks to Lemma \ref{lemacomparacion},
    \begin{equation*}
        u(x,t) \leq M, \;\;\; v(x,t) \leq N
    \end{equation*}
    for all $(x,t)\in\overline{\Omega}\times [0,\tau)$.
\end{proof}

The comparison result from Lemma \ref{lemacomparacion} also gives us some a priori estimates of the solutions of \eqref{1.1} that we will use later.

\begin{cor} \label{lemamenorque0}
Let $(u,v)$ be a solution of \eqref{1.1} with initial data $(u_0,v_0)$. If the initial data satisfy, for every $x\in\overline\Omega$,
$$
\begin{array}{l}
\displaystyle\int_\Omega J(x-y)u_0(y)\,dy
+\int_{\mathbb{R}^N\setminus\Omega} J(x-y)\,dy -u_0(x)-\lambda v_0^{-p}(x)\le 0\\
\displaystyle\int_\Omega J(x-y)v_0(y)\,dy
+\int_{\mathbb{R}^N\setminus\Omega} J(x-y)\,dy -v_0(x)-\mu u_0^{-q}(x)\le0
\end{array}
$$
then $u_t(x,t)\leq 0$ and $v_t(x,t) \leq 0$ for every $x\in\overline{\Omega}$ and every $t \in[0,\tau)$, with $\tau>0$ the maximal existence time of the solution.
\end{cor}

\begin{proof}
Notice that, since $u$ and $v$ are $\mathcal{C}^2$ in the time variable, $(u_t,v_t)$ satisfies the following equations in $\overline\Omega \times [0,\tau)$:
\begin{align*}
    (u_t)_t (x,t) &= \int_\Omega J(x-y) u_t(y,t) dy - u_t(x,t) +\lambda p v^{-p-1} v_t (x,t), \\
    (v_t)_t (x,t) &= \int_\Omega J(x-y) v_t(y,t) dy - v_t(x,t) +\mu q u^{-q-1} u_t (x,t),
\end{align*}
and the initial data fulfill $u_t(x,0) \leq 0$, $v_t(x,0) \leq 0$ thanks to the hypotheses.

This system of equations is of the type \eqref{auxsistema.general}, with $I_{u_t}=I_{v_t}=\mathbb{R}$, $h_1(x,t)=h_2(x,t)=-1$, $g_1(x,t,u_t,v_t) = \lambda p v(x,t)^{-p-1} v_t$ and $g_2(x,t,u_t,v_t) = \mu q u(x,t)^{-q-1} u_t$.

It is easy to see that $(0,0)$ is a supersolution of this system and $g_1, g_2$ fulfill the conditions of Lemma \ref{lemacomparacion}, so we can apply it to get the desired result.
\end{proof}

\end{section}

\begin{section}{Quenching versus global existence}

In this section, we will prove Theorem \ref{teo.estacionarias}. To do so, first we need some lemmas that will help us understand how the solutions of the system behave. Throughout this section we will assume $\lambda,\mu,p,q>0$ unless otherwise specified.

\begin{lema} \label{condicionquenching}
    Let $u_0,v_0\in \mathcal{C}(\overline{\Omega})$ be positive initial data of \eqref{1.1} such that at some point $x_0\in \overline{\Omega}$ they satisfy
    \begin{equation} \label{condquenching}
        u_0(x_0) \leq \left( \frac{\mu}{N+\varepsilon} \right)^{1/q}, \; \text{ and } \;\;\; v_0(x_0) \leq \left( \frac{\lambda}{M+\varepsilon} \right)^{1/p},
    \end{equation}
    where $M=\max\{1,\|u_0\|_\infty\}$, $N=\max\{1,\|v_0\|_\infty\}$ and $\varepsilon>0$. Then the solution of system \eqref{1.1} presents quenching in finite time
    $T_\varepsilon\le \min\{u_0(x_0),v_0(x_0)\}/\varepsilon$.
    Furthermore, $u_t(x_0,t) < -\varepsilon$ and $v_t(x_0,t) < -\varepsilon$ for every $t\in[0,T_\varepsilon)$.
\end{lema}
\begin{proof}
First we note that the solution $(u(x,t),v(x,t))$ is defined in $(x,t) \in \overline\Omega \times [0,T_\varepsilon)$, where $T_\varepsilon \in (0,\infty)$ or $T_\varepsilon=+\infty$. Thanks to Lemma \ref{lema.MNsuper},
    \begin{equation*}
        u(x,t) \leq M, \;\;\; v(x,t) \leq N
    \end{equation*}
for all $(x,t)\in\overline\Omega\times[0,T_\varepsilon)$.  On the other hand, by \eqref{condquenching} and the fact that $u_0$ and $v_0$ are strictly positive functions,
    \begin{equation*}
        u_t(x_0,0) < \int_\Omega J(x_0-y) u_0(y) + \int_{\mathbb{R}^N \backslash \Omega}J(x-y) dy - \lambda v_0(x_0)^{-p} \leq M - \lambda \left(\frac{\lambda}{M+\varepsilon} \right)^{-1} =  -\varepsilon.
    \end{equation*}
    \begin{equation*}
        v_t(x_0,0) < \int_\Omega J(x_0-y) v_0(y) + \int_{\mathbb{R}^N \backslash \Omega}J(x-y) dy - \mu u_0(x_0)^{-q} \leq N - \mu \left(\frac{\mu}{N+\varepsilon} \right)^{-1} = -\varepsilon
    \end{equation*}

    Now suppose that there exists a time $t_1\in(0,T_\varepsilon)$ in which either $u_t(x_0,t)$ or $v_t(x_0,t)$ reach $0$ for the first time. Without loss of generality, $u_t(x_0,t_1)=0$, $u_t(x_0,t)<0$ for all $t \in [0, t_1)$, and $v_t(x_0,t) < 0 $ for all $t\in[0,t_1)$.
    In this situation, it is clear that $v(x_0,t)<v_0(x_0)$ for $t\in[0,t_1)$. Then we conclude that
    \begin{align*}
       u_t(x_0,t_1) &= \int_\Omega J(x_0-y) u(y,t_1) + \int_{\mathbb{R}^N \backslash \Omega}J(x_0-y) dy - u(x_0,t_1) - \lambda v(x_0,t_1)^{-p}, \\
       &< \int_\Omega J(x_0-y) u(y,t_1) + \int_{\mathbb{R}^N \backslash \Omega}J(x_0-y) dy - \lambda v(x_0,t_1)^{-p} < M - \lambda v_0(x_0)^{-p} \leq -\varepsilon
    \end{align*}
    which brings us to a contradiction with the fact that $u_t(x_0,t_1)=0$. Therefore, $u(x_0,\cdot)$ and $v(x_0,\cdot)$ are decreasing functions and applying the same argument as before,
    $$
    u_t(x_0,t) <-\varepsilon,\qquad v_t(x_0,t) <-\varepsilon
    $$
    for every $t\in[0,T_\varepsilon)$. Integrating these inequalities, we get
    $$
    u(x_0,t)<u_0(x_0)-\varepsilon t,\qquad
    v(x_0,t)<v_0(x_0)-\varepsilon t.
    $$
    Therefore, the solution quenches at finite time $T_\varepsilon$ and we have
    $$
    T_\varepsilon\le \min\left\{\frac{u_0(x_0)}{\varepsilon},\frac{v_0(x_0)}{\varepsilon}\right\}.
    $$
\end{proof}

The following result gives us an important property on the behaviour of globally defined solutions of system \eqref{1.1}.

\begin{lema} \label{lema.globalesmenorque1}
    Let $(u,v)$ be a globally defined solution of system \eqref{1.1}. Then there exists a time $t_1 \in [0,\infty)$ for which $u(x,t)\leq 1$ and $v(x,t)\leq 1$ for every $x\in\overline\Omega$ and $t\in[t_1,\infty)$. Furthermore, if $(w,z)$ is a stationary solution of \eqref{1.1}, then $\mu^{1/q} < w(x) \leq 1$ and $ \lambda^{1/p} < z(x) \leq 1$ for every $x\in\overline\Omega$.
\end{lema}

\begin{proof}
    Consider $(u,v)$ a globally defined solution with initial data $(u_0,v_0)$.

    If $\|u_0\|_\infty, \|v_0\|_\infty \leq 1$, then $(U,V)$ with $U \equiv V \equiv 1$ is a supersolution of \eqref{1.1} with these initial data. Using Lemma \ref{lemacomparacion}, we get that $1 \geq u(x,t)$ and $1\geq v(x,t)$ for every $(x,t)\in\overline\Omega \times [0,\infty)$.

    If $\|u_0\|_\infty > 1$, consider the pair of functions
    \begin{equation*}
     U(t) = \|u_0\|_\infty - \lambda \|v_0\|_\infty^{-p} t, \;\;\; V(t) = \|v_0\|_\infty - C t,
    \end{equation*}
    with $0<C\leq \mu \|u_0\|_\infty^{-q}$ sufficiently small so that there is a time $t_0 \in [0,\infty)$ in which $U(t_0)=1$ and $V(t_0)\geq 1$. Then we can prove that $(U,V)$ is a supersolution of \eqref{1.1} in $[0,t_0)$. Indeed, we have that $U(0)\geq \|u_0\|_\infty$, $V(0)\geq \|v_0\|_\infty$ and, since $U(t),V(t)\geq 1$ for every $t\in[0,t_0]$:
\begin{equation*}
    U(t) \int_\Omega J(x-y) dy + \int_{\mathbb{R}^N\backslash\Omega} J(x-y) dy - U(t) - \lambda V(t)^{-p} \leq - \lambda \|v_0\|_\infty^{-p} = U_t (t),
\end{equation*}
\begin{equation*}
    V(t) \int_\Omega J(x-y) dy + \int_{\mathbb{R}^N\backslash\Omega} J(x-y) dy - V(t) - \mu U(t)^{-q} \leq - \mu \|u_0\|_\infty^{-q} \leq -C = V_t (t).
\end{equation*}
Then by Lemma \ref{lemacomparacion} we know that $u(x,t) \leq U(t)$ and $v(x,t) \leq V(t)$ for every $x\in\overline\Omega$ and $t\in[0,t_0)$. Since $u(\cdot,t),v(\cdot,t),U(t),V(t)$ are continuous in $t\in [0,\infty)$, it is also true that $u(x,t_0) \leq U(t_0) =  1$ and $v(x,t_0) \leq V(t_0)$. Now consider the pair of functions
\begin{equation*}
     \overline{U} (t) = 1, \;\;\; V (t) = \|v_0\|_\infty - C t,
    \end{equation*}
and define $t_1 \in [t_0,\infty)$ as the time in which $V(t_1)=1$. Then we can prove that $(\overline{U}, V)$ is a supersolution of \eqref{1.1} in $[t_0,t_1)$. Indeed, we know that $\overline{U}(t_0) \geq u(x,t_0)$ and $V(t_0) \geq v(x,t_0)$ for every $x\in\overline\Omega$, and, for every $t\in[t_0,t_1)$:
\begin{equation*}
    \int_\Omega J(x-y) dy + \int_{\mathbb{R}^N\backslash\Omega} J(x-y) dy - 1 - \lambda V(t)^{-p} \leq - \lambda \|v_0\|_\infty^{-p} \leq 0 = \overline{U}_t (t),
\end{equation*}
\begin{equation*}
    V(t) \int_\Omega J(x-y) dy + \int_{\mathbb{R}^N\backslash\Omega} J(x-y) dy - V(t) - \mu \leq - \mu \leq -C = V_t (t),
\end{equation*}
where we have used that $C \leq \mu \|u_0\|_\infty^{-q} \leq \mu$ because $\|u_0\|_\infty > 1$. Then by Lemma \ref{lemacomparacion}, $u(x,t) \leq 1$ and $v(x,t) \leq V(t)$ for every $x\in\overline\Omega$ and $t\in[t_0,t_1)$. Again by continuity, $v(x,t_1) \leq V(t_1) = 1$.
Since there is a time $t_1 \in [0,\infty)$ in which $u(x,t_1) \leq 1$ and $v(x,t_1) \leq 1$ for every $x\in\overline\Omega$, we can consider $(\overline{U},\overline{V})$ with $\overline{U} \equiv \overline{V} \equiv 1$ as a supersolution of \eqref{1.1} for every $t\in[t_1,\infty)$, which by Lemma \ref{lemacomparacion}, gives us that $u(x,t) \leq 1$ and $v(x,t) \leq 1$ for every $x\in\overline\Omega$ and $t\in[t_1,\infty)$.

If $\|v_0\|_\infty > 1$, the proof is analogous.

Let us prove now the final statement of the lemma. Let $(w,z)$ be a stationary solution of \eqref{1.1}. Suppose first that $\|w\|_\infty >1$ or $\|z\|_\infty>1$. This is a globally defined solution and, following the previous argument, we reach a contradiction. Suppose now that there exists some $x_1 \in \overline\Omega$ such that $w(x_1) \leq \mu^{1/q}$. Since $(w,z)$ is a stationary solution, we have that
\begin{equation*}
    1 > \int_\Omega J(x_1-y)z(y) dy + \int_{\mathbb{R}^N \backslash \Omega} J(x_1-y) dy - z (x_1) = \mu w^{-q}(x_1) \geq 1.
\end{equation*}
which is a contradiction. If there exists $x_2\in\overline\Omega$ such that $z(x_2) \leq \lambda^{1/p}$, the argument is analogous and we reach a contradiction as well, so the result is proven.
\end{proof}

Thanks to this lemma, we get an important relation between the solution with constant initial data of value $1$ and the stationary solutions of the system.

\begin{lema}\label{lema.1quencheatodo}
Let $(u_1,v_1)$ be the solution of system \eqref{1.1} with initial data $u(x,0) \equiv v(x,0) \equiv 1$. Then, the solution $(u_1,v_1)$ satisfies either:

(i) It is well-defined for all $t\in[0,\infty)$ and it converges uniformly in space to a stationary solution $(w,z)$, or

(ii) it presents quenching in finite time. In this case all solutions present quenching in finite time.

Moreover, in case (i), the convergence to the stationary solution is from above, that is, $u_1(x,t) \geq w(x)$ and $v_1(x,t)\geq z(x)$ for every $(x,t)\in\overline\Omega \times [0,\infty)$.
\end{lema}
\begin{proof}
    $(i)$ First we assume that $(u_1,v_1)$ is a global solution. By Corollary \ref{lemamenorque0} we know that $(u_1)_t(x,t) \leq 0, (v_1)_t(x,t) \leq 0$ for every $x\in\overline\Omega$ and $t\geq0$. Since $(u_1,v_1)$ is bounded from below (by zero), its components converge point-wise, that is, $u_1 (x,t) \xrightarrow{t\rightarrow\infty} w(x)$ and $v_1 (x,t) \xrightarrow{t\rightarrow\infty} z(x)$ in $\overline\Omega$.

    First, let us prove that $w$ and $z$ are strictly positive. Indeed, suppose that there is $\overline{x} \in \overline\Omega$ such that $w(\overline{x})=0$ (the argument is the same if $z(\overline{x})=0$). This means that, for every $\varepsilon>0$, there exists a time $t_\varepsilon \in [0,\infty)$ such that $u_1 (\overline{x}, t) < \varepsilon$ for every $t\geq t_\varepsilon$. Taking $\varepsilon< \mu^{1/q}$, using Lemma \ref{lema.MNsuper}, and noticing that, in this case, $M=N=1$, we get that
    \begin{equation*}
        (v_1)_t (\overline{x},t) \leq 1 - v_1(\overline{x},t) - \mu \varepsilon^{-q} \leq 1 - \mu \varepsilon^{-q} < 0
    \end{equation*}
    for every $t\geq t_\varepsilon$. Since $u_1(\overline{x},t) \xrightarrow{t\to\infty} 0$ and $v_1(\overline{x},t)$ is strictly decreasing in $[t_\varepsilon,\infty)$, there exists $\tilde{t} \geq t_\varepsilon$ such that $u_1(\overline{x},\tilde{t}) \leq (\mu/2)^{1/q}$ and $v_1(\overline{x},\tilde{t}) \leq (\lambda/2)^{1/p}$. Then $(u_1,v_1)$ is under the conditions of Lemma \ref{condicionquenching} and the solution presents quenching in finite time, which is a contradiction.

    We then want to show that $(w,z)$ is a stationary solution. For that, bearing in mind that $w$ and $z$ are strictly positive, we can take limits on every term of \eqref{1.1} and use the Dominated Convergence Theorem to get, for every $x\in\overline\Omega$,
    \begin{align*}
        \lim_{t\rightarrow\infty} (u_1)_t(x,t) &= \int_\Omega J(x-y)w(y) dy + \int_{\mathbb{R}^N \backslash \Omega} J(x-y) dy - w(x) - \lambda z^{-p}(x). \\
        \lim_{t\rightarrow\infty} (v_1)_t(x,t) &= \int_\Omega J(x-y)z(y) dy + \int_{\mathbb{R}^N \backslash \Omega} J(x-y) dy - z(x) - \mu w^{-q}(x).
    \end{align*}
    Since $u_1(x,\cdot)$ and $v_1(x,\cdot)$ are bounded continuous functions in $\mathbb{R}^+$ and $\lim_{t\to\infty} (u_1)_t(x,t)$, $\lim_{t\to\infty} (v_1)_t(x,t)$ exist, these limits must be $0$, and we conclude $(w,z)$ is a stationary solution.

    Let us now prove that the convergence is uniform. First, we will prove that $(w,z)$ are continuous functions in $\overline\Omega$. Indeed, since $u_1$ and $v_1$ are non-increasing in time, for some fixed $x_0\in \overline\Omega$ and $t_0\in[0,\infty)$, we know that
    \begin{equation*}
        \begin{pmatrix}
        ((u_1)_t)_t (x_0,t) \\ ((v_1)_t)_t(x_0,t)
        \end{pmatrix} \leq \begin{pmatrix}
            -1  & \lambda p v_1^{-p-1}(x_0,t_0) \\
            \mu q u_1^{-q-1} (x_0,t_0) & -1 
        \end{pmatrix}
        \begin{pmatrix}
        (u_1)_t (x_0,t) \\ (v_1)_t(x_0,t)
        \end{pmatrix}.
    \end{equation*}
    for every $t\in[t_0,\infty)$. Furthermore, we know that $(u_1)_t(x_0,t),(v_1)_t(x_0,t) \leq 0$ and both tend to $0$ as $t\to\infty$, so by comparison the equilibrium point $(0,0)$ of the system is stable, and the eigenvalues of the square matrix must be both negative. This means that
    \begin{equation} \label{condicion.positividad.continuidad}
        1 - \lambda\mu pq v_1^{-p-1}(x_0,t_0) u_1^{-q-1}(x_0,t_0) > 0.
    \end{equation}
    Notice that this is true for every $(x_0,t_0) \in \overline\Omega \times[0,T)$. Next, thanks to $(w,z)$ being stationary solutions and the continuity of the integro-differential operator, we know that both $(w+\lambda z^{-p})(x)$ and $(z+\mu w^{-q})(x)$ are continuous in $\overline\Omega$. Suppose without loss of generality that $w(x)$ is not continuous at some point $\hat{x} \in \overline\Omega$, that is, there exists some $\{x_n\}_n \subset \overline\Omega$ such that
    \begin{equation*}
        x_n \longrightarrow \hat{x}, \;\; w(x_n) \longrightarrow A > w(\hat{x}), \;\; z(x_n) \longrightarrow B,
    \end{equation*}
    as $n\to \infty$ with some $A,B>0$ (the argument is very similar if $A<w(\hat{x})$). Using the continuity of the previous functions and applying the Mean Value Theorem, we get
    \begin{align*} \label{inequality.assuming.discontinuity}
        0 =& \, (A-w(\hat{x})) + \lambda(B^{-p} -z^{-p}(\hat{x})) >  (A-w(\hat{x})) - \lambda p z^{-p-1}(\hat{x})(B -z(\hat{x}))\\
        =& \,(A-w(\hat{x})) + \lambda p z^{-p-1}(\hat{x})(A^{-q} - w^{-q}(\hat{x})) > (A-w(\hat{x}))(1- \lambda \mu pq z^{-p-1}(\hat{x}) w^{-q-1}(\hat{x})).\nonumber
    \end{align*}
    This means that $1- \lambda \mu pq z^{-p-1}(\hat{x}) w^{-q-1}(\hat{x}) < 0$, which is a contradiction with \eqref{condicion.positividad.continuidad}.
    
    Finally, since $(u_1,v_1)(x,t)$ are continuous, non-increasing in $t$ and they converge pointwise to $(w,z)(x)$, which are also continuous, the convergence is uniform via Dini's Theorem.

    $(ii)$ Assume now that $(u_1,v_1)$ presents quenching in finite time $T$ and let us suppose that there exists $(u,v)$ a global solution of \eqref{1.1}. By Lemma \ref{lema.globalesmenorque1}, there exists a time $t_1\in[0,\infty)$ such that $u(x,t) \leq 1$ and $v(x,t) \leq 1$ for every $t\in[t_1,\infty)$. In particular, $u(x,t_1)\leq 1$ and $v(x,t_1) \leq 1$. Therefore, $(u_1(x,t),v_1(x,t))$ is a supersolution of system \eqref{1.1} with initial data $(u(x,t_1),v(x,t_1))$. By Lemma \ref{lemacomparacion}, $u_1(x,t) \geq u(x,t+t_1)$ and $v_1(x,t) \geq v(x,t+t_1)$ for every $t\in[0,T)$. Since $(u_1(x,t),v_1(x,t))$ quenches at time $T$, $(u,v)$ will also quench at some time $\widetilde T\leq T+t_1$, which is a contradiction with the assumption that $(u,v)$ is a global solution. Therefore, if $(u_1,v_1)$ presents quenching then every solution of \eqref{1.1} also quenches.
\end{proof}

\begin{lema} \label{lema.existenciaestacionarias}
    There is a neighbourhood of $(0,0)$, $U^\prime \subset \mathbb{R}^2$, such that there exist stationary solutions of system \eqref{1.1}  if $(\lambda,\mu) \in  U^\prime \cap ((0,\infty) \times (0,\infty))$.
\end{lema}

\begin{proof}
    This result follows from the Implicit Function Theorem. First we note that, if $\lambda = \mu = 0$, then $w_0\equiv z_0 \equiv 1$ is a stationary solution. We linearize around this solution, considering:
    \begin{equation*}
        \phi = 1 - u, \;\;\; \gamma = 1 - v.
    \end{equation*}
    Given small $\delta,\varepsilon >0$, let
    \begin{equation*}
    Y_0 = \{ (\phi,\gamma) \in \mathcal{C}(\overline{\Omega}) \times \mathcal{C}(\overline{\Omega}) : -\delta < \phi,\gamma < 1   \}
    \end{equation*}
    and  $\mathcal{F}: (-\varepsilon,\varepsilon) \times (-\varepsilon,\varepsilon) \times Y_0 \longrightarrow \mathcal{C}(\overline{\Omega}) \times \mathcal{C}(\overline{\Omega})$ be a non-linear operator given by
    \begin{align*}
        \mathcal{F}(\lambda,\mu,\phi,\gamma)(x) =& \left(\int_{\Omega} J(x-y)\phi(y) dy - \phi(x) + \lambda (1-\gamma(x))^{-p}, \right.\\
        &\left. \int_{\Omega} J(x-y)\gamma(y) dy - \gamma(x) + \mu (1-\phi(x))^{-q} \right).
    \end{align*}
    It is clear that $\mathcal{F}(0,0,0,0) = (0,0)$ and the differential operator with respect to the last two variables $D_{(\phi,\gamma)} \mathcal{F}$ evaluated at $\lambda=\mu=0$ is, for every $\xi,\chi\in \mathcal{C}(\overline\Omega)$ and $x\in\overline\Omega$,
    \begin{equation*}
        D_{(\phi,\gamma)} \mathcal{F} (0,0)(\xi,\chi)(x) = \begin{pmatrix}
                                        (K_J - I)(\xi)(x) & 0 \\
                                        0 & (K_J - I)(\chi)(x)
                                     \end{pmatrix},
    \end{equation*}
    where $K_J: \mathcal{C}(\overline\Omega) \longrightarrow \mathcal{C}(\overline\Omega)$ is defined as
    \begin{equation*}
        K_J (z)(x) = \int_\Omega J(x-y) z(y) dy.
    \end{equation*}

    The operator $K_J - I$ is clearly continuous and linear. Let us prove that it is also an injective operator.

    Take $z\in\mathcal{C}(\overline\Omega)$ such that $K_J (z) = z$ and let us prove that $z\equiv 0$. Define $x_0$ as $z(x_0) = \max_{x\in\overline\Omega} z(x)$ and take some $\delta>0$ such that $B(0,\delta) \subset \text{supp}(J)$. Suppose that $z(x_0)>0$. Under these assumptions we know that
    \begin{equation*}
        z(x_0) = K_J(z)(x_0)= \int_\Omega J(x_0-y)z(y)dy \leq \int_\Omega J(x_0-y)z(x_0)dy \leq z(x_0).
    \end{equation*}
    Therefore,
    \begin{equation*}
        \int_\Omega J(x_0-y) (z(y)-z(x_0)) dy = 0.
    \end{equation*}
    Since $J(x_0-y) > 0$ for every $y \in B(x_0,\delta)$, we conclude that $z(y) = z (x_0)$ for every $y\in B(x_0,\delta)$. The domain $\overline\Omega$ is connected and compact, so for any $\overline{x}\in\overline\Omega$, there exists a chain of points $\{x_0,x_1,x_2, \dots, x_n=\overline{x}\}$ such that $x_i \in B(x_{i-1},\delta)$ for every $i=1,\dots,n$. We know that $z(x_1)=z(x_0)$ thanks to our previous argument and, since $\delta$ only depends on $J$, it is clear that $z(x_2)=z(x_1)$ because $x_2 \in B(x_1,\delta)$. Repeating this, we get that $z(x_0) = z(x_i)$ for every $i=1,\dots,n$ and, in particular, $z(\overline{x})=z(x_0)$. Therefore, $z(x)=z(x_0)$ for every $x\in\overline\Omega$. However, if $z$ is a constant function then $K_J(z)=z$ implies that $\int_\Omega J(x-y) dy = 1$ for every $x\in\overline\Omega$, which is a contradiction. Then we conclude that $z(x)\leq 0$ for every $x\in\overline\Omega$. Applying the same argument to the function $-z$, we obtain that $z(x)\geq 0$ for every $x\in\overline\Omega$, which implies that $z\equiv 0$. Therefore, $K_J - I$ is an injective operator.

    Additionally, $K_J$ is compact, see Section 2.1.3 of \cite{SG}. Using Fredholm alternative we can further assert that $K_J - I$ is bijective and $D_{(\phi,\gamma)} \mathcal{F}(0,0)$ is too. Finally, using the Open Mapping Theorem, we deduce that $D_{(\phi,\gamma)}\mathcal{F}(0,0): \mathcal{C}(\overline{\Omega}) \times \mathcal{C}(\overline{\Omega}) \longrightarrow \mathcal{C}(\overline{\Omega}) \times \mathcal{C}(\overline{\Omega})$ is an isomorphism.

    Therefore, we can apply the Implicit Function Theorem to ensure that there is a neighbourhood $U^\prime \subset \mathbb{R}^2$ of $(0,0)$ such that, if $(\lambda,\mu) \in U^\prime$, then there is a $(\phi_{\lambda,\mu},\gamma_{\lambda,\mu})\in Y_0$ that satisfies $\mathcal{F}(\lambda,\mu,\phi_{\lambda,\mu},\gamma_{\lambda,\mu}) = (0,0)$, see \cite{CR}. Finally, if we define
    \begin{equation*}
    w_{\lambda,\mu}=\left\{
    \begin{array}{ll}
    1-\phi_{\lambda,\mu}, & x\in\overline\Omega,\\
    1, & x\in \mathbb{R}^N\backslash \overline\Omega,
    \end{array}\right.
    \qquad
    z_{\lambda,\mu}=\left\{
    \begin{array}{ll}
    1-\gamma_{\lambda,\mu}, & x\in\overline\Omega,\\
    1, & x\in \mathbb{R}^N\backslash \overline\Omega,
    \end{array}\right.
    \end{equation*}
    these are stationary solutions of our system with parameters $(\lambda,\mu) \in U^\prime \cap ((0,\infty) \times (0,\infty))$.
\end{proof}

With these lemmas we can finally prove one of our main results.

\begin{proof}[Proof of Theorem \ref{teo.estacionarias}]
    The proof is given in several steps.

    $i)$ There is a small neighbourhood around the origin $U^\prime \subset \mathbb{R}^2$ in which the system accepts stationary solutions, thanks to Lemma \ref{lema.existenciaestacionarias}.

    $ii)$ If we have $\lambda \geq 1$, there can be no stationary solutions in our system. This result follows from Lemma \ref{lema.globalesmenorque1}. If there exists a stationary solution $(w,z)$, we know that $\lambda^{1/p}<z(x)\leq 1$ for every $x\in\overline\Omega$, which implies that $\lambda < 1$. If we have $\mu\geq1$, it can also be proven that there are no stationary solutions with the same reasoning.

    $iii)$ If $(w_0,z_0)$ is a stationary solution of the system with parameters $(\lambda_0,\mu_0)$, then there exists a stationary solution for every $(\lambda,\mu)$ with $\lambda\leq\lambda_0$ and $\mu\leq \mu_0$.

    For this, we take a pair of parameters $(\lambda,\mu)$ such that $\lambda\leq \lambda_0$ and $\mu\leq\mu_0$, and notice that $(w_0,z_0)$ is a subsolution of the system \eqref{1.1} with parameters $(\lambda,\mu)$. Define $(u_1,v_1)$ as the solution of \eqref{1.1} with parameters $(\lambda,\mu)$ and initial data $u(x,0)\equiv v(x,0) \equiv 1$. Then, due to Lemma \ref{lemacomparacion}, $(u_1,v_1)$ is bounded from below by $(w_0,z_0)$. Thanks to Lemma \ref{lema.1quencheatodo}, $(u_1,v_1)$ converges to a stationary solution $(w,z)$. Moreover
    $$
    w(x)\ge w_0(x),\qquad z(x)\ge z_0(x),
    $$
    which gives us a monotonicity property of the stationary solutions with respect to the parameters $(\lambda,\mu)$.

    $iv)$ Consider now a fixed $\lambda$ and define
    \begin{equation*}
    \mu_\lambda^* = \sup\{ \mu : \text{a stationary solution exists for } (\lambda,\mu) \}.
    \end{equation*}
    Due to Lemma \ref{lema.globalesmenorque1}, we know that any stationary solution $(w,z)$ of system \eqref{1.1} with parameters $(\lambda,\mu)$ fulfills, for every $x\in\overline\Omega$:
    \begin{equation*}
        1 \geq w(x) > \mu^{1/q}, \;\;\; 1 \geq z(x) > \lambda^{1/p}.
    \end{equation*}
    Thanks to this fact and the monotonicity property proved in $iii)$,  it is clear that, for every $x\in\overline\Omega$:
    \begin{equation*}
        w_{\lambda,\mu_\lambda^*} (x) = \lim_{\mu\rightarrow \mu_\lambda^*} w_{\lambda,\mu} (x) > 0, \;\;\; z_{\lambda,\mu_\lambda^*} (x) = \lim_{\mu\rightarrow \mu_\lambda^*} z_{\lambda,\mu} (x) > 0.
    \end{equation*}
    Then taking the limit $\mu \to \mu_\lambda^\ast$ in the equations of \eqref{1.1}, we see that $(w_{\lambda,\mu_\lambda^*}, z_{\lambda,\mu_\lambda^*} )$ is a stationary solution of \eqref{1.1} with parameters $(\lambda, \mu^\ast_\lambda)$.

    We conclude that there exists a neighbourhood $U$ of $(0,0)$ in $\mathbb{R}^2$ such that there are stationary solutions if and only if $(\lambda, \mu) \in \overline{U} \cap ((0,\infty) \times (0,\infty))$.

    $v)$ Finally, since there are no stationary solutions for $(\lambda,\mu) \notin \overline{U} \cap ((0,\infty) \times (0,\infty))$, every solution quenches by Lemma \ref{lema.1quencheatodo}.
\end{proof}

\begin{rem}
    We do not know the full geometry of the set $U$ in general, only the properties we have outlined in Theorem \ref{teo.estacionarias}. We do not even know the behaviour of the stationary solutions when we approach the boundary of $U$, where some kind of bifurcation is expected. This is also an open problem in the particular case of the system with only one equation.
\end{rem}

\end{section}

\begin{section}{Simultaneous and Non-simultaneous quenching}

In this section we will prove Theorem~\ref{teo.simultaneo}. We consider two different cases: $\max\{p,q\}\ge 1$ and $\max\{p,q\}<1$.
For the first one, let $(u,v)$ be a solution of \eqref{1.1} that presents quenching at time $T<+\infty$. Multiplying the first equation by $\mu u^{-q}(x,t)$ and the second equation by $\lambda v^{-p}(x,t)$, we get
\begin{align*}
  \mu  u^{-q} u_t (x,t) &= \mu u^{-q} \left( J*u-u \right)(x,t) - \lambda \mu u^{-q} v^{-p}(x,t), \\
    \lambda v^{-p} v_t (x,t) &= \lambda v^{-p} \left( J*v-v\right)(x,t) - \lambda \mu u^{-q} v^{-p}(x,t),
\end{align*}
for every $(x,t)\in\overline\Omega \times[0,T)$. Therefore,
$$
\mu u^{-q} u_t (x,t)- \lambda v^{-p} v_t (x,t)= \mu u^{-q} \left( J*u-u \right)(x,t)- \lambda v^{-p} \left( J*v-v\right)(x,t),
$$
for every $(x,t)\in\overline\Omega \times [0,T)$. If we integrate the term on the left hand side between $0$ and $t\in[0,T)$, we get
$$
\begin{array}{rl}
\displaystyle\left|\int_0^t \Big(\mu u^{-q} u_t (x,s)- \lambda v^{-p} v_t (x,s)\Big) ds\right| \ge &
\left|(\mu \Psi_q[u]-\lambda\Psi_p[v])(x,t)\right|- \left|(\mu\Psi_q[u]-\lambda\Psi_p[v])(x,0)\right|\\
\ge & \left|(\mu\Psi_q[u]-\lambda\Psi_p[v])(x,t)\right|- D,
\end{array}
$$
for every $(x,t)\in\overline\Omega \times [0,T)$, where $D>0$ and $\Psi_a[g]$ is a primitive of the function $g^{-a}g_t(x,\cdot)$, that is,
\begin{equation}\label{primitiva}
\Psi_a[g](x,t)=\left\{
\begin{array}{ll}
\frac{g^{1-a}}{1-a}(x,t) & a\ne1\\
\ln |g(x,t)|\qquad & a=1.
\end{array}\right.
\end{equation}
Therefore, we know that
\begin{align*}
    \left|\mu\Psi_q[u](x,t)-\lambda\Psi_p[v](x,t)\right| &\leq D + \left|\int_0^t \Big(\mu u^{-q} u_t (x,s)- \lambda v^{-p} v_t (x,s)\Big) ds\right| \\
    &\leq D + \left|\int_0^t (\mu u^{-q} \left( J*u-u \right)(x,s)+ \lambda v^{-p} \left( J*v-v \right)(x,s))ds\right|,
\end{align*}
for every $(x,t)\in\overline\Omega \times [0,T)$. Moreover, we know that  $0<\delta\le \min\{u_0(x),v_0(x)\}\le C$ and Lemma \ref{lema.MNsuper} implies that $u(x,t)\le M$ and $v(x,t)\le N$ for every $(x,t) \in \overline\Omega \times [0,T)$. This gives us that
$$
\begin{array}{rl}
\displaystyle\left|\int_0^t \mu u^{-q} \left( J*u-u \right)(x,s)ds\right|\le &
\displaystyle2 M \int_0^t \mu u^{-q} (x,s)ds
= 2M \int_0^t (J*v-v-v_t)(x,s)ds\\
\le &\displaystyle 4MN (T+1),
\end{array}
$$
and
$$
\begin{array}{rl}
\displaystyle\left|\int_0^t \lambda v^{-p} \left( J*v-v\right)(x,s)ds\right|\le &
\displaystyle2 N \int_0^t \lambda v^{-p} (x,s)ds
= 2N \int_0^t (J*u-u-u_t)(x,s)ds\\
\le &\displaystyle 4MN (T+1),
\end{array}
$$
for every $(x,t)\in\overline\Omega \times [0,T)$. Finally,
\begin{equation}\label{des.clave}
\left|\mu \Psi_q[u](x,t)-\lambda\Psi_p[v](x,t)\right|\le D + 8MN(T+1) = \tilde{D},
\end{equation}
for every $(x,t)\in\overline\Omega \times [0,T)$. Notice that, for $p,q<1$, this estimate is empty because  the primitives $\Psi_q[u]$ and $\Psi_p[v]$ are bounded functions. However, for the case $\max\{p,q\}\ge1$, at least one of the primitives is unbounded and we obtain a lot of information from the inequality.

\begin{lema}\label{max(p,q)>=1}
Let $\max\{p,q\}\ge 1$ and $(u,v)$ be a solution of system \eqref{1.1} that presents quenching at time $T<+\infty$. Then,
\begin{enumerate}[i)]
\item If $p\ge 1>q$, then only the component $u$ presents quenching.
\item If $q\ge 1>p$, then only the component  $v$ presents quenching.
\item If $p, q\ge 1$, and for some sequence $(x_n,t_n)\to (x_0,T)$ we have that one of the components tends to zero, then the other component also tends to zero.
    Therefore, the quenching is always simultaneous and $Q(u)=Q(v)$.
\end{enumerate}
\end{lema}
\begin{proof}
Let $p=\max\{p,q\}\ge1$ (the case $q=\max\{p,q\}\ge1$ is similar).

Assume first that  $p\ge1>q$ and suppose that $x_0\in Q(v)$. Then there exists a sequence $\{(x_n,t_n)\}_n$ such that $(x_n,t_n)\to (x_0,T)$ and $v(x_n,t_n)\to 0$ as $n\to\infty$. Then
$$
\Psi_p[v](x_n,t_n)\to -\infty,\qquad \Psi_q[u](x_n,t_n)\le C,
$$
because $q<1$ and thus the second primitive is bounded for any values of $u(x_n,t_n)$. This is a contradiction with the inequality \eqref{des.clave}. So, for $p\ge1>q$, the component $v$ cannot present quenching and the quenching is always non-simultaneous. The argument for the case $q\geq 1 > p$ is analogous.

Now assume $p,q\ge 1$ and take some $x_0\in Q(u,v)$. Then there exists a  sequence $\{(x_n,t_n)\}_n$ such that $(x_n,t_n)\to (x_0,T)$ and one of the components tends to zero as $n\to\infty$. Without loss of generality, $u(x_n,t_n)\to 0$ as $n\to\infty$. As $\Psi_q[u](x_n,t_n)$  is unbounded, inequality \eqref{des.clave} implies that $\Psi_p[v](x_n,t_n)$ is also unbounded.
Then, $v(x_n,t_n)\to 0$, the quenching is always simultaneous and $Q(u)=Q(v)$.
\end{proof}

\begin{cor}
    Let $(u,v)$ be a quenching solution of system $\eqref{1.1}$ that presents simultaneous quenching. Then either $p,q\geq 1$ or $p,q<1$.
\end{cor}

Observe that estimate \eqref{des.clave} also gives us the following results.
\begin{lema}
Let $(u,v)$ be a solution of system \eqref{1.1} that presents quenching and $x_0\in Q(u,v)$. Then,
\begin{enumerate}[i)]
\item If only the component $u$ presents quenching at $x_0$, then $q<1$.
\item If only the component $v$ presents quenching at $x_0$, then $p<1$.
\end{enumerate}
\end{lema}
\begin{lema}\label{lema-estimaciones}
Let $p,q\ge1$ and $\{(x_n,t_n)\}_n$ a sequence such that both components present quenching as $n\to\infty$. Then,
\begin{enumerate}[i)]
\item for $p,q>1$, $u^{1-q}(x_n,t_n)\sim v^{1-p}(x_n,t_n)$;
\item for $p>1=q$, $-\log(u(x_n,t_n))\sim v^{1-p}(x_n,t_n)$;
\item for $q>1=p$, $u^{1-q}(x_n,t_n)\sim -\log(v(x_n,t_n))$;
\item for $p=q=1$, $u^{\mu}(x_n,t_n)\sim v^\lambda(x_n,t_n)$.
\end{enumerate}
\end{lema}

The case $p,q<1$ is more involved because the non-simultaneous and simultaneous quenching coexist. To prove that we use a shooting argument. Let $(u_0,v_0)$ be two functions such that
\begin{equation} \label{condicion.quenching.iv2.bis}
        \|u_0\|_\infty \leq \min \left\{ 1, \left( \frac{\mu}{2} \right)^{1/q} \right\}, \;\;\; \|v_0\|_\infty \leq \min \left\{ 1, \left( \frac{\lambda}{2} \right)^{1/p} \right\}.
    \end{equation}
and let $(u_\delta,v_\delta)$  be the solution of system \eqref{1.1} with initial data $(\delta u_0,(1-\delta)v_0)$ for $\delta\in(0,1)$. Notice that the hypotheses of Lemma \ref{condicionquenching} are satisfied for all $x\in\overline\Omega$ with $M=N=\varepsilon=1$. Then the solution presents quenching at finite time
\begin{equation}\label{Tdelta}
T_\delta \leq \min \left\{ \min_{x\in\overline\Omega}\delta u_0(x), \min_{x\in\overline\Omega} (1-\delta) v_0(x) \right\},
\end{equation}
and
\begin{equation}\label{eq.decrece}
(u_\delta)_t(x,t)\le -1,\qquad  (v_\delta)_t(x,t)\le -1,
\end{equation}
for every $(x,t)\in\overline\Omega \times [0,T_\delta)$. Integrating these inequalities between $t\in[0,T_\delta)$ and $T_\delta$,
\begin{equation}\label{eq.tasainf}
u_\delta (x,t)\ge (T_\delta-t), \qquad
v_\delta (x,t)\ge (T_\delta-t),
\end{equation}
for every $(x,t)\in\overline\Omega \times [0,T_\delta)$. Now, let us introduce the following disjoint sets:
\begin{equation*}
    \begin{array}{l}
        \displaystyle A^+=\{\delta\in(0,1)\,:\, \mbox{only the component $v_\delta$ presents quenching} \},\\
        \displaystyle A^-=\{\delta\in(0,1)\,:\, \mbox{only the component $u_\delta$ presents quenching}\},\\
        \displaystyle A=\{\delta\in(0,1)\,:\, \mbox{both components present quenching}\}.
    \end{array}
\end{equation*}
If we prove that both $A^{+}$ and $A^{-}$ are nonempty open sets, then it will follow from the connectedness of the interval $(0,1)$ that $A$ is a closed and nonempty subset of $(0,1)$. Therefore, in this case we would have some initial datum with simultaneous quenching taking $\delta\in A$ and initial data with non-simultaneous quenching taking $\delta \in A^+ \cup A^-$.

\begin{lema} \label{lema.sets.nonempty}
Let $p,q<1$. The sets $A^+$ and $A^-$ are nonempty.
\end{lema}
\begin{proof}
Notice that  by \eqref{eq.tasainf}
$$
\begin{array}{rl}
(u_\delta)_t (x,t)= &\displaystyle\int_\Omega J(x-y) u_\delta (y,t) + \int_{\mathbb{R}^N \backslash\Omega} J(x-y) dy -u_\delta (x,t)-\lambda v_\delta^{-p} (x,t)\\
\ge&\displaystyle   -u_\delta (x,t) -\lambda (T_\delta -t)^{-p},
\end{array}
$$
for every $(x,t) \in \overline\Omega \times [0,T_\delta)$. Solving the differential inequality, we get
$$
e^t u_\delta(x,t)\ge  u_\delta(x,0)- \lambda\int_0^t e^s(T_\delta-s)^{-p}ds\ge
 \delta u_0(x)-\frac{\lambda}{1-p} e^{T_\delta} ( T_\delta^{1-p}-(T_\delta-t)^{1-p}),
$$
for every $(x,t) \in \overline\Omega \times [0,T_\delta)$. Therefore,
$$
\lim_{t\nearrow T_\delta} e^t u_\delta(x,t)\ge \delta u_0(x)- \frac{\lambda}{1-p} e^{T_\delta} T_\delta^{1-p},
$$
for every $x\in\overline\Omega$. Since $T_\delta\to 0$ as $\delta\to 1$, we can take $\delta$ sufficiently close to $1$ to get that
$$
\lim_{t\nearrow T_\delta} e^t u_\delta(x,t) \ge \frac\delta2 u_0(x)>0,
$$
for every $x\in\overline\Omega$. Therefore only the component $v_\delta$ quenches and the set $A^+$ is nonempty.

With the same argument, it is easy to see that taking $\delta$ sufficiently close to $0$ only the component $u$ quenches and the set $A^-$ is nonempty.
\end{proof}

To prove that $A^+$ and $A^-$ are open sets, we first need to see that the quenching time is continuous with respect to $\delta$.

\begin{lema}
    Let $p,q<1$ and consider $(u_\delta,v_\delta)$ the solution of system \eqref{1.1} with initial data $(\delta u_0,(1-\delta) v_0)$.  Then the quenching time $T_\delta$ of this solution is continuous with respect to $\delta$.
\end{lema}

\begin{proof}
We take $\delta$ and $\widetilde\delta$ such that $|\delta-\widetilde\delta|\le\mu$ and define the function
\begin{equation*}
m(t) = \min \left\{ \min_{x\in\overline\Omega} u_{\delta}(x,t),  \min_{x\in\overline\Omega} v_{\delta}(x,t),
\min_{x\in\overline\Omega} u_{\widetilde\delta}(x,t),  \min_{x\in\overline\Omega} v_{\widetilde\delta}(x,t) \right\}
\end{equation*}
for $t\in (0,T_0)$ with $T_0=\min\{T_\delta,T_{\widetilde\delta}\}$.
Observe that $m(t)\to 0$ as $t\to T_0$ and from \eqref{eq.decrece} it is a decreasing function. So there exists a time $t_0$ such that $m(t_0)=\varepsilon/3$. Moreover assuming that  $m(t_0)=\min_{x\in\overline\Omega} u_{\delta}(x,t_0)=u_{\delta}(x_0,t_0)$, we use the continuity of the solutions with respect to the initial data in the system \ref{1.1} to obtain that $u_{\widetilde\delta}(x_0,t_0)\le 2\varepsilon/3$ provided that $\mu$ is small enough.
Finally, by \eqref{eq.tasainf} we have
$$
|T_\delta-T_{\widetilde\delta}|\le |T_\delta-t_0|+|T_{\widetilde\delta}-t_0|\le u_{\delta}(x_0,t_0)+u_{\widetilde\delta}(x_0,t_0)\le \varepsilon
$$
and the result follows.
\end{proof}

\begin{lema} \label{lema.sets.open}
    Let $p,q<1$. Then $A^+$  and $A^-$ are open sets.
\end{lema}
\begin{proof}
We only study $A^+$, the study of $A^-$ is similar. Notice that for $\delta_0 \in A^+$  there exists $c>0$ such that $u_{\delta_0}(x,t)\ge c$ for every $(x,t) \in \overline\Omega \times [0,T_{\delta_0})$.

Now for a fixed $\varepsilon>0$, we consider $\hat{t} \in (T_{\delta_0} - \varepsilon/2,T_{\delta_0}-\varepsilon/4)$.
By the continuity of the quenching time with respect to $\delta$ there exists some $\mu_1>0$ such that $|T_\delta - T_{\delta_0}| < \varepsilon/4$ provided $|\delta-\delta_0| < \mu_1$. Notice that $\hat t<T_\delta$, in fact
$$
T_\delta-\hat t\le|T_\delta-T_{\delta_0}|+|T_{\delta_0}-\hat t|\le \varepsilon.
$$
Notice that as $\hat t<T_\delta$, the function $u_\delta(x,\hat t)$ is well defined. Then due to
the continuous dependence with respect to the initial data, there exists $\mu_1 \geq \mu_2>0$ such that $u_\delta(x,\hat{t}) \geq c/2$ for every $x\in\overline\Omega$ if $|\delta-\delta_0| < \mu_2$.

Moreover, by \eqref{eq.tasainf},
$$
\begin{array}{rl}
(u_\delta)_t (x,t)= &\displaystyle\int_\Omega J(x-y) u_\delta (y,t) + \int_{\mathbb{R}^N \backslash\Omega} J(x-y) dy -u_\delta (x,t)-\lambda v_\delta^{-p} (x,t)\\
\ge&\displaystyle   -u_\delta (x,t) -\lambda (T_\delta -t)^{-p},
\end{array}
$$
for any $\delta \in (0,1)$ and every $(x,t)\in\overline\Omega \times [0,T_\delta)$. Now consider $\delta \in (\delta_0 - \mu_2, \delta_0 + \mu_2)$. Integrating between $\hat{t}$ and $ T_\delta$,
$$
\begin{array}{rl}
\displaystyle \lim_{t\nearrow {T}_\delta} e^t u_\delta(x,t)
&\displaystyle
\ge e^{\hat{t}} u_\delta(x,\hat{t})- \frac{\lambda}{1-p} e^{T_\delta}( T_\delta-\hat{t})^{1-p}
\ge e^{T_\delta}\Big(\frac{ce^{-\varepsilon}}{2} - \frac{\lambda}{1-p} \varepsilon^{1-p}\Big).
\end{array}
$$
Then taking $\varepsilon$ small enough we deduce that $u_\delta$ does not quench, that is, $(\delta_0-\mu_2,\delta_0+\mu_2)\subset A^+$ and the result follows.
\end{proof}

\begin{proof}[Proof of Theorem \ref{teo.simultaneo}]
    The result follows from Lemma \ref{max(p,q)>=1}, Lemma \ref{lema.sets.nonempty} and Lemma \ref{lema.sets.open}.
\end{proof}

\end{section}

\begin{section}{Quenching Rates}
To study the quenching rates in each region of the plane $(p,q)$, we will need to work with the quantities
$$
\min_{x \in \overline\Omega} u(\cdot,t) = u(x_u(t),t), \qquad \min_{x \in \overline\Omega} v(\cdot,t) = v(x_v(t),t).
$$
The following result proves that both quantities are differentiable for almost every time.
\begin{lema} \label{lema.minlipschitz}
    Let $u\in\mathcal{C}^1((0,T), \mathcal{C}(\overline\Omega))$, and define
    \begin{equation*}
        m(t) = \min_{x\in\overline\Omega} u(\cdot,t) = u(x (t),t).
    \end{equation*}
    Then $m(t)$ is differentiable and $m^\prime (t) := \partial_t (u(x(t),t)) = u_t (x(t),t)$ for almost every time $t\in(0,T)$.
\end{lema}
\begin{proof}
    Take $0<\tau<t<T$ and use the fact that $x(t)$ is the minimum of $u(x,t)$ at time $t$ to get that
    \begin{equation*}
        m(t)-m(\tau) = u(x (t),t) - u(x (\tau),\tau)
    \end{equation*}
    satisfies
    $$
    u(x (t),t) - u(x (t),\tau)\le m(t)-m(\tau) \leq u(x (\tau),t) - u(x(\tau),\tau).
    $$
    Then applying the Mean Value Theorem over $u$,
    \begin{equation} \label{ineq.minimum}
        u_t(x (t),\xi_2) \leq \frac{m(t)-m(\tau)}{t-\tau} \leq u_t(x (\tau),\xi_1),
    \end{equation}
    with $\xi_1,\xi_2\in[\tau,t]$. Since $u_t(x,\cdot)$ is locally bounded, we get that $m(\cdot)$ is locally Lipschitz. Therefore $m(\cdot)$ is differentiable for almost every time and passing to the limit as $\tau\to t$ in \eqref{ineq.minimum}:
    $$
m'(t):=  \partial_t (u(x (t),t)) = u_t (x (t),t).
    $$
\end{proof}

\subsection{Non-simultaneous quenching}
In this subsection we prove Theorem \ref{teo.tasas.nosimultanea}. Assume that $u$ is the component that presents quenching (the other case is similar), that is,
$$
\liminf_{t\to T} u(x_u(t),t)=0, \qquad
v(x,t)\ge \delta \text{   for } (x,t) \in \overline\Omega \times [0,T).
$$
By  Theorem \ref{teo.existe} we have that $\lim_{t\to T} u(x_u(t),t)=0$. Then there exists a sequence $t_n\to T$ such that
\begin{equation}\label{cero-nosimultaneo}
u(x_u (t_n),t_n)\to 0, \qquad
u_t(x_u (t_n),t_n)=-c_n<0.
\end{equation}
Moreover, $u_t(x,t)$ satisfies
\begin{equation} \label{ut.bounds}
-M-\lambda \delta^{-p} \le u_t (x,t) =J*u(x,t) -u(x,t) -\lambda v^{-p}(x,t) \le M,
\end{equation}
for every $(x,t)\in\overline\Omega \times [0,T)$. Thanks to Lemma \ref{lema.minlipschitz}, we can integrate the left inequality on \eqref{ut.bounds} between $t\in[0,T)$ and $T$ to obtain the upper quenching rate
\begin{equation} \label{upper.qrate.nonsim}
u(x_u(t),t)\le (M+\lambda\delta^{-p}) (T-t).
\end{equation}
Hence, given $\varepsilon>0$, we have
$$
u(x_u(t),t)\le \varepsilon\qquad \mbox{for } t\in(t_\varepsilon,T),
$$
where $t_\varepsilon$ is defined by  $(M+\lambda\delta^{-p}) (T-t_\varepsilon)=\varepsilon$. Moreover, taking $t_\varepsilon<s<t<\widetilde s<T$ we can integrate the left inequality in \eqref{ut.bounds} between $s$ and $t$ to get
$$
u(x_u(t),s)\le
u(x_u (t),t)+(M+\lambda\delta^{-p}) (t - s) \leq u(x_u(t),t) + (M+\lambda\delta^{-p}) (T - t_\varepsilon)=2\varepsilon
$$
and integrate the upper inequality between $t$ and $\widetilde s$ to get
$$
u(x_u(t),\widetilde s)\le u(x_u(t),t) + M (\widetilde s-t) \leq u(x_u(t),t) + (M+\lambda\delta^{-p}) (T - t_\varepsilon) =2\varepsilon.
$$
Therefore we conclude that
$$
u(x_u(t),\tau) \leq 2\varepsilon \qquad \mbox{for } t,\tau \in (t_\varepsilon,T).
$$

Now let us observe that for a fixed $x\in\overline\Omega$ we can differentiate the equation of $u$ to get
\begin{equation} \label{ineq.utt}
\begin{array}{rl}
u_{tt}(x,t)\leq &\displaystyle \int_\Omega J(x-y)u_t(y,t)dy-u_t(x,t)+\lambda p v^{-p-1}( J*v-\mu u^{-q})(x,t)\\
\le & C_1- C_2 u^{-q}(x,t),
\end{array}
\end{equation}
for every $(x,t)\in\overline\Omega \times [0,T)$.
Thus,
$$
u_{tt}(x(t),\tau)<0,\qquad \mbox{for } t,\tau \in(t_\varepsilon,T)
$$
provided that  $2\varepsilon<(C_2/C_1)^{1/q}$. On the other hand, taking $t_\varepsilon < \tau < t < T$, we can follow the proof of Lemma \ref{lema.minlipschitz} to get
$$
u_t(x_u (t),s_2) \leq \frac{u(x_u (t),t)-u
(x_u (\tau),\tau)}{t-\tau} \leq u_t(x_u (\tau),s_1),
$$
with $s_1,s_2\in[\tau,t]$. Now, using the fact that both $u_t(x_u(t),s)$ and $u_t(x_u(\tau),s)$ are strictly decreasing for $s\in(t_\varepsilon,T)$, we have
$$
u_t(x_u (t),t)<u_t(x_u (t),s_2)\le u_t(x_u (\tau),s_1)<u_t(x_u (\tau),\tau),
$$
Finally, by \eqref{cero-nosimultaneo}, we can take $\tau=t_n\in(t_\varepsilon,T)$ to arrive at
$$
u_t(x_u (t),t) < u_t(x_u (t_n),t_n) = -c_n,
$$
for every $t\in[t_n,T)$. By integration we obtain the lower quenching rate. \qed

\subsection{Simultaneous quenching}
In order to prove the simultaneous quenching rate, we need some relation between the two components of system \eqref{1.1}. Notice that estimate \eqref{des.clave} gives us this relation in the case $p,q\ge1$, see Lemma \ref{lema-estimaciones}. However, in the case $p,q<1$, the estimate \eqref{des.clave} does not give us any information because if $g(x_n,t_n)\to 0$ as $n\to\infty$, the function $\Psi_a[g](x_n,t_n)\to 0$ whenever $a<1$. Hence, to give the necessary relation  between the components in this second case we need to impose extra hypotheses in the initial data.

\begin{lema}\label{todoacero}
    Let $p,q\geq 1$ and $(u,v)$ be a solution of \eqref{1.1} that presents quenching at time $T<+\infty$. Then
    \begin{equation*}
        \lim_{t\nearrow T} u(x_u(t),t) =  \lim_{t\nearrow T} v(x_v(t),t) = 0.
    \end{equation*}
\end{lema}
\begin{proof}
By Lemma \ref{max(p,q)>=1} we get that the quenching is simultaneous and
$$
\liminf_{t\nearrow T} u(x_u(t),t)=0,\qquad \liminf_{t\nearrow T} v(x_v(t),t)=0.
$$
Now, arguing by contradiction we suppose that $\limsup u(x_u(t),t)=2c>0$. Then there exists a sequence $\{t_n\}_n$ with $t_n \xrightarrow{n\to\infty} T$ and such that $u(x_u(t_n),t_n)\ge  c>0$ for every $n\in\mathbb{N}$. Clearly, $u(x_v(t_n),t_n)\ge u(x_u(t_n),t_n)\ge  c$ for every $n\in\mathbb{N}$ and using Lemma \ref{max(p,q)>=1} we get that $v(x_v(t_n),t_n)\ge  \widetilde c>0$. Therefore,
$$
\min \left\{ \min_{x\in\overline\Omega} u(x,t_n), \min_{x\in\overline\Omega} v(x,t_n)\right\} \ge \min\{c,\widetilde c\}>0,
$$
for every $n\in\mathbb{N}$, which is a contradiction with Theorem \ref{teo.existe}.
\end{proof}

\begin{proof}[Proof of Theorem \ref{lema.qrate.sim}]
Observe that the absorption terms in system \eqref{1.1} are unbounded whenever $u$ or $v$ present quenching while the diffusion terms are always bounded, so there exists a time $t_0\in[0,T)$ such that
\begin{equation} \label{equiv.ut}
    u_t(x_u(t),t) \sim - v^{-p} (x_u(t),t), \qquad
    v_t(x_v(t),t) \sim - u^{-q} (x_v(t),t)
\end{equation}
for every $t\in[t_0,T)$.

 $i)$ Assume that $p,q > 1$. From Lemma \ref{lema-estimaciones} we know that there exists $t_1\in [t_0,T)$ such that
\begin{equation}\label{equiv.uv1}
    u^{1-q} (x_u (t),t) \sim v^{1-p} (x_u (t),t),\qquad
    u^{1-q} (x_v (t),t) \sim v^{1-p} (x_v (t),t)
\end{equation}
for every $t\in[t_1,T)$. To study the behaviour of the solutions along the sequence $(x_u(t),t)$ we use the first equivalence in \eqref{equiv.ut}, to get
\begin{equation*}
    u_t(x_u(t),t) \sim - u^{\frac{pq-p}{1-p}}(x_u(t),t),
\end{equation*}
Note that, thanks to Lemma \ref{lema.minlipschitz}, we know that $u_t(x_u(t),t) = \partial_t (u(x_u(t),t))$. Then we can integrate this equivalence between $t\in[t_1,T)$ and $T$ and arrive to the desired quenching rate
\begin{equation*}
    u(x_u(t),t) \sim (T-t)^{\frac{p-1}{pq-1}}.
\end{equation*}
The quenching rate  for the component $v$ is given by \eqref{equiv.uv1}.
\begin{equation*}
    v(x_u(t),t) \sim (T-t)^{\frac{q-1}{pq-1}}.
\end{equation*}

To consider the behaviour of the solution along the sequence $(x_v(t),t)$ we observe that
$$
v(x_v(t),t)\le v(x_u(t),t)\le C (T-t)^{\frac{q-1}{pq-1}},
$$
and
$$
C (T-t)^{\frac{p-1}{pq-1}}\le u(x_u(t),t) \le u(x_v(t),t)\le C v^{\frac{p-1}{q-1}}(x_v(t),t)
$$
for every $t\in[t_1,T)$, which gives us the desired quenching rate
\begin{equation*}
    v(x_v(t),t) \sim (T-t)^{\frac{q-1}{pq-1}}.
\end{equation*}
As before, the quenching rate  for the component $u$ is given by \eqref{equiv.uv1}.

$ii)$ Assume now that $p> 1= q$ and consider the behaviour of the solution along the sequence $(x_u(t),t)$ (the quenching rate along $(x_u(t),t)$ gives us the quenching rate along $(x_v(t),t)$ with the same reasoning as before).

In this case, Lemma \ref{lema-estimaciones} gives us that there exists $t_2\in [t_0,T)$ such that
\begin{equation}\label{equiv.uv2}
    v^{1-p} (x_u(t),t) \sim -\log u(x_u(t),t) = \log \left(\frac{1}{u(x_u(t),t)} \right)
\end{equation}
for every $t\in[t_2,T)$.  Then by \eqref{equiv.ut},
\begin{equation}\label{eq.log}
  u_t(x_u(t),t) \sim - \left(\log \frac1{u(x_u(t),t)}\right)^{\frac{p}{p-1}}.
\end{equation}
Integrating this expression between $t\in[t_2,T)$ and $T$, the rate is given implicitly by
\begin{equation} \label{integral.rate}
\int_{u(x_u(t),t)}^0 \left( \log \left( \frac1{s} \right) \right)^{\frac{p}{1-p}}\,ds\sim -(T-t).
\end{equation}
We can then apply the L'H\^opital rule over the following limit to get
\begin{equation}\label{limitelog}
    \lim_{t\nearrow T} \frac{\int^{u(x_u(t),t)}_{0} \log (1/s)^{\frac{p}{1-p}} ds}{ u(x_u(t),t) \log (1/u(x_u(t),t))^{\frac{p}{1-p}}}=1.
\end{equation}
Therefore, using both \eqref{integral.rate} and \eqref{limitelog}, there exists $t_3\in[t_2,T)$ such that
$$
\log \left(\frac{1}{u(x_u(t),t)}\right)^{\frac{p}{p-1}}\sim \frac{u(x_u(t),t)}{(T-t)}
$$
for every $t\in[t_3,T)$. Putting these estimates in \eqref{eq.log}
\begin{equation*}
    u_t(x_u(t),t) u^{-1}(x_u(t),t) \sim -(T-t)^{-1},
\end{equation*}
expression that we can integrate between $t_3$ and $t \in [t_3,T)$ to obtain
\begin{equation} \label{equiv.logu.logT}
    \log u(x_u(t),t) \sim \log(T-t),
\end{equation}
for every $t\in[t_3,T)$. Using again this estimate in \eqref{eq.log} we get
\begin{equation*}
    u_t(x_u(t),t) \sim -(-\log (T-t))^{\frac{p}{p-1}}.
\end{equation*}
Then,
$$
u(x_u(t),t)\sim \int_t^T (-\log(T-s))^{\frac{-p}{1-p}} ds
= \int_{-\log(T-t)}^\infty \tau^{\frac{-p}{1-p}} e^{-\tau} d\tau.
$$
That is, $u(x_u(t),t)$ behaves like the upper incomplete Gamma function with the following parameters:
\begin{equation*}
    u(x_u(t),t) \sim \Gamma\left( \frac{1-2p}{1-p}, -\log (T-t) \right).
\end{equation*}
Using the asymptotic properties of the upper incomplete Gamma function, we know that:
\begin{equation*}
    \frac{\Gamma\left( \frac{1-2p}{1-p}, -\log (T-t) \right)}{(-\log (T-t))^{-p/(1-p)}(T-t)} \xrightarrow{t\rightarrow T} 1,
\end{equation*}
and the quenching rate for $u$ follows. The quenching rate for the $v$ component follows immediately from
\eqref{equiv.logu.logT} and  \eqref{equiv.uv2}.

$iii)$ The proof for this case is analogous to that of point $ii)$.

$iv)$ The proof for this case is analogous to that of point $i)$.
\end{proof}

\begin{proof}[Proof of Theorem \ref{lema.qrate.sim.menor1}]
    Since the quenching is simultaneous,
    $$
    \liminf_{t\to T} u(x(t),t)=0, \qquad
    \liminf_{t\to T} v(x(t),t)=0.
    $$
    Suppose then that $\limsup_{t\to T} u(x(t),t)=c>0$. Since, for $t$ close to $T$, the function $u(x(\cdot),\cdot)$ oscillates between $(0,c)$, there exists a sequence of times $t_n\to T$ such that
    $$
    u(x(t_n),t_n)= c/2,\qquad
    u_t(x(t_n),t_n)\ge0.
    $$
    This is a contradiction because, from Theorem \ref{teo.existe}, we would have that $v(x(t_n),t_n)\to 0$ which implies that $u_t(x(t_n),t_n)\to -\infty$. The same happens if we suppose that $\limsup_{t\to T} v(x(t),t)=c>0$. Then we have that
    $$
    \lim_{t\to T} u(x(t),t)=0, \qquad
    \lim_{t\to T} v(x(t),t)=0.
    $$

    Now, we note that
    $$
    \lim_{t\nearrow T} \frac{v^pu_t}{u^qv_t}(x(t),t) =
    \lim_{t\nearrow T} \frac{v^p (J*u-u)-\lambda}{u^q(J*v-v)-\mu}(x(t),t) =
    \frac{\lambda}{\mu}.
    $$
    This implies that there exists $t_2\in[t_0,T)$ such that
    $$
    u^{-q}u_t(x(t),t)\sim v^{-p}v_t(x(t),t).
    $$
    Integrating between $t\in[t_2,T)$ and $T$ we get the estimate
    $$
    u^{1-q}(x(t),t)\sim v^{1-p}(x(t),t).
    $$
    At this point we can follow the same argument as in the case $p,q>1$ in Theorem \ref{lema.qrate.sim} to obtain the quenching rate.
\end{proof}

\end{section}
\begin{section}{Numerical simulations}

In this section we illustrate numerically some of the properties proved before for problem \eqref{1.1}. We take a uniform partition of size $h=4/(N+1)$ of $\Omega=[-2,2]$, that is, $-2=x_{-N}<\cdots<x_0=0<\cdots<x_N=2$. We consider the semidiscrete  approximation $(u_i(t),v_i(t))\sim (u(x_i,t),v(x_i,t))$ which solves the following ODE system
$$
u_i'=h\sum_{j=-N}^N J(x_i-x_j) u_j+\Big(1-h\sum_{j=-N}^N J(x_i-x_j) \Big)-u_i-\lambda v_i^{-p},
$$
$$
v_i'=h\sum_{j=-N}^N J(x_i-x_j) v_j+\Big(1-h\sum_{j=-N}^N J(x_i-x_j) \Big)-v_i-\mu u_i^{-q},
$$
where we have chosen $J(x) = \frac{3}{4}(1-x^2)_+$ as the kernel, and we will take multiple sets of parameters $\lambda,\mu,p,q>0$. To perform the integration in time we use an adaptive ODE solver for stiff problems (Matlab ode23s).
In all cases the blue color represents the $u$ component and the red color the $v$ component. We take $N=100$ and $u_i(0)=1=v_i(0)$.

In Figure \ref{estacionarias} we look for the existence of stationary solutions, see Theorem \ref{teo.estacionarias}. Notice that for $\lambda=0.001=\mu$ the solution stabilizes to a stationary solution while for $\lambda=0.1=\mu$ we have simultaneous quenching at time $T=9.0619$ and the only quenching point is the origin.

\begin{figure}[ht] \centering
\includegraphics[width=7.2cm]{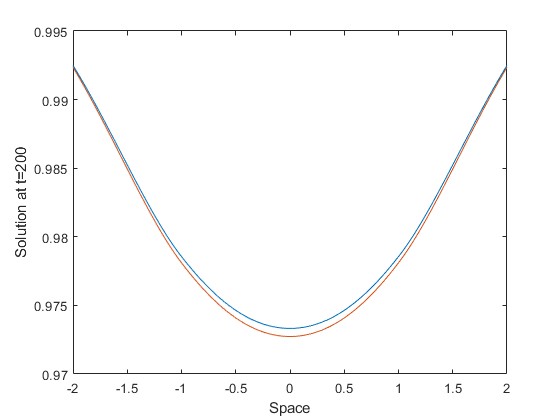}
\qquad
\includegraphics[width=7.2cm]{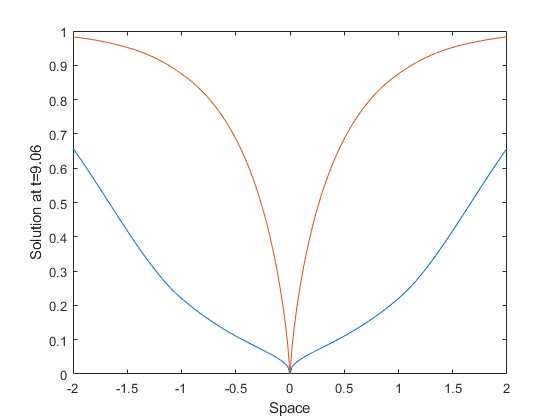}
\caption{$p=2$ and $q=3$. To the left, $\lambda=0.001=\mu$. To the right, $\lambda=0.1,\mu = 0.001$.}
\label{estacionarias}
\end{figure}

In order to see the quenching rate we note that, by Theorem \ref{lema.qrate.sim},
$$
u(0,t)\sim (T-t)^{\frac{p-1}{pq-1}}=(T-t)^{\frac15},\qquad
v(0,t)\sim (T-t)^{\frac{q-1}{pq-1}}=(T-t)^{\frac25}.
$$
In Figure \ref{tasa_simultanea} we represent the $(-\log(u(0,t)),-\log(v(0,t)))$ versus $-\log(T-t)$ next to dashed lines of slopes $1/5$ and $2/5$.

\begin{figure}[H] \centering
\includegraphics[width=7.2cm]{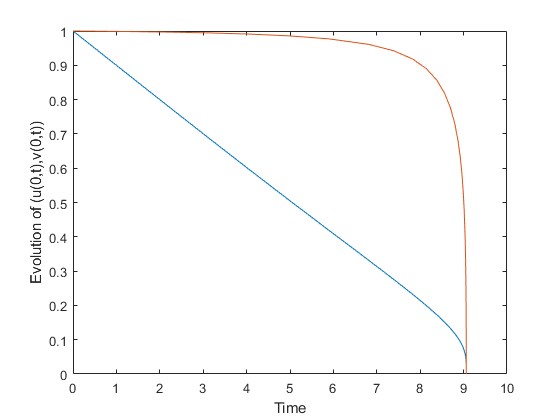}
\qquad
\includegraphics[width=7.2cm]{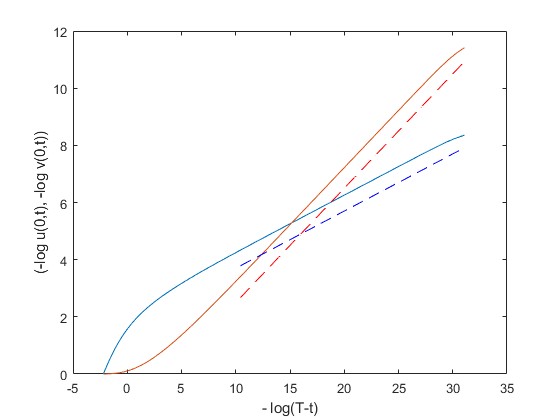}
\caption{$\lambda=0.1, \mu = 0.001$, $p=2$ and $q=3$. To the left, the evolution of $(u(0,t),v(0,t))$. To the right, the simultaneous quenching rate.}
\label{tasa_simultanea}
\end{figure}

Finally, we look for the non-simultaneous quenching. Taking $p=2$ and $q=0.7$, Theorems \ref{teo.simultaneo} and \ref{teo.tasas.nosimultanea} say that only the $u$ component presents quenching and $u(0,t)\sim (T-t)$. Similarly, taking $p=0.2$ and $q=3$, only the $v$ component presents quenching and $v(0,t) \sim (T-t)$. This is also the case of our approximation, see Figures \ref{pmayor1} and \ref{pmayor2}.

\begin{figure}[ht] \centering
\includegraphics[width=7.2cm]{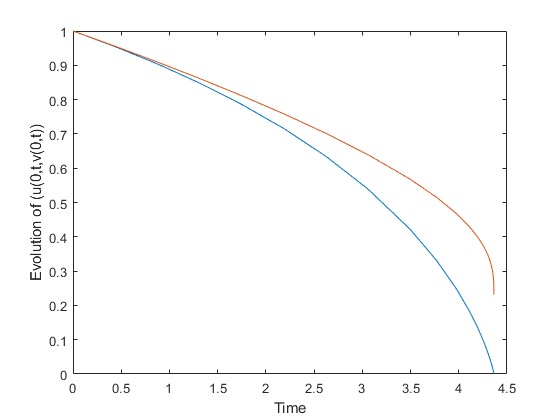}
\qquad
\includegraphics[width=6.8cm]{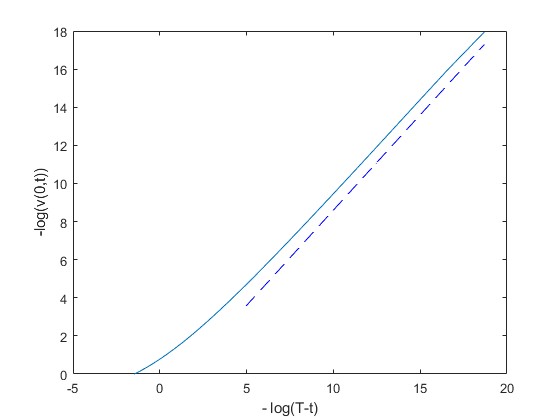}
\caption{$\lambda=0.1=\mu$, $p=2$ and $q=0.7$. To the left, the evolution of $(u(0,t),v(0,t))$. To the right, the quenching rate.}
\label{pmayor1}
\end{figure}

\begin{figure}[ht] \centering
\includegraphics[width=7.2cm]{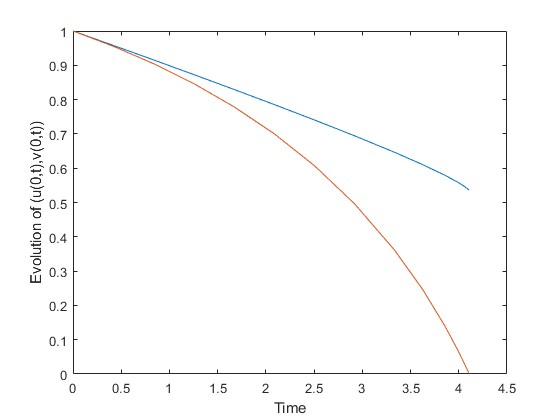}
\qquad
\includegraphics[width=7.2cm]{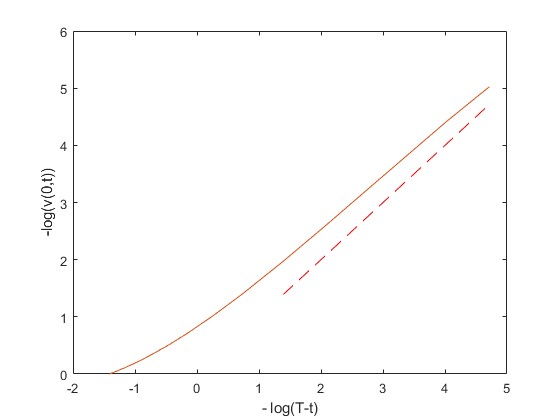}
\caption{$\lambda=0.1=\mu$, $p=0.2$ and $q=3$. To the left, the evolution of $(u(0,t),v(0,t))$. To the right, the quenching rate.}
\label{pmayor2}
\end{figure}

\end{section}

\newpage

\centerline{{\bf Acknowledgements}}

J. M. Arrieta is partially supported by grants  PID2022-137074NBI00 and CEX2023-001347-S “Severo Ochoa Programme for Centres of Excellence in R\&D”, the two of them funded by 
MCIU/AEI/10.13039/501100011033, Spain.  Also by “Grupo de Investigación 920894 - CADEDIF”, UCM, Spain.

R. Ferreira is partially supported by the Spanish project PID2023-146931NB-I00 funded by the MICIN/AEI and by “Grupo de Investigación 920894 - CADEDIF”, UCM, Spain.

S. Junquera is  partially supported by grant  PID2022-137074NBI00 funded by \linebreak MCIU/AEI/10.13039/501100011033, Spain and by the FPU21/05580 grant from the Ministry of Science, Innovation and Universities.

\vspace{10mm}

\centerline{{\bf Disclosure of interest}}

The authors report there are no competing interests to declare.

\newpage

\section*{Figures}

\setcounter{figure}{0}

\begin{figure}[ht] \centering
\includegraphics[width=7.7cm]{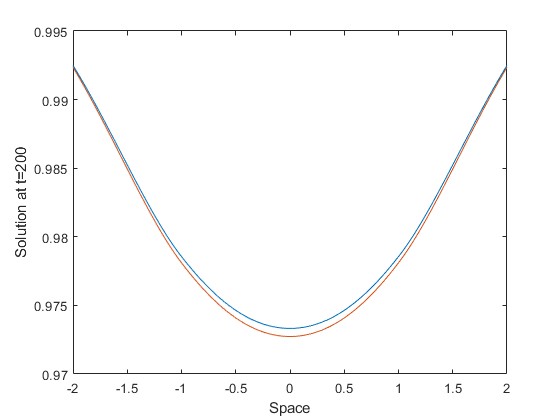}
\qquad
\includegraphics[width=7.7cm]{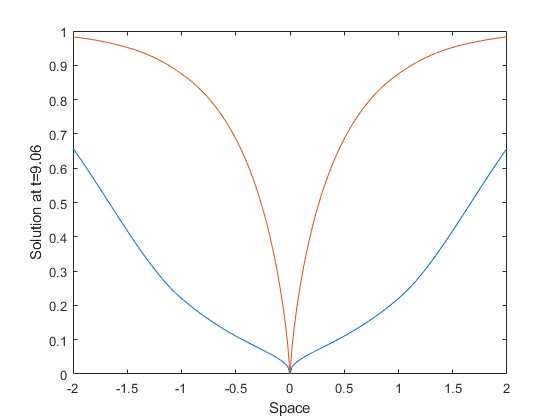}
\caption{$p=2$ and $q=3$. To the left, $\lambda=0.001=\mu$. To the right, $\lambda=0.1,\mu = 0.001$.}
\end{figure}

\begin{figure}[ht] \centering
\includegraphics[width=7.7cm]{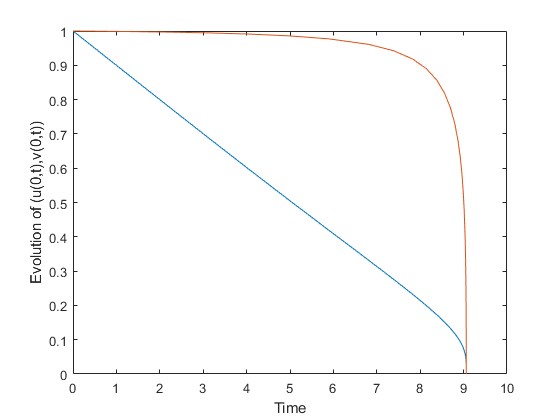}
\qquad
\includegraphics[width=7.7cm]{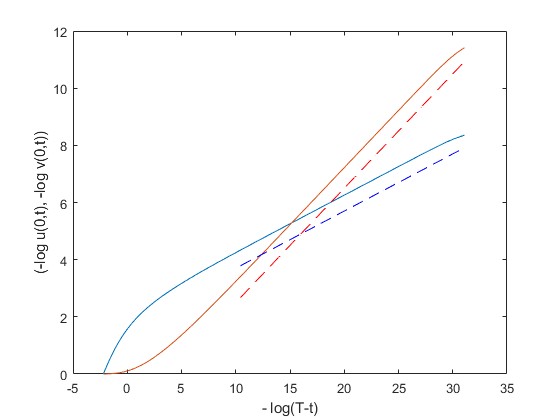}
\caption{$\lambda=0.1, \mu = 0.001$, $p=2$ and $q=3$. To the left, the evolution of $(u(0,t),v(0,t))$. To the right, the simultaneous quenching rate.}
\end{figure}

\begin{figure}[ht] \centering
\includegraphics[width=7.7cm]{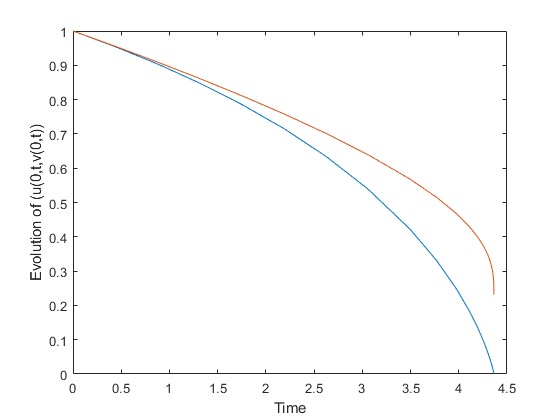}
\qquad
\includegraphics[width=7.7cm]{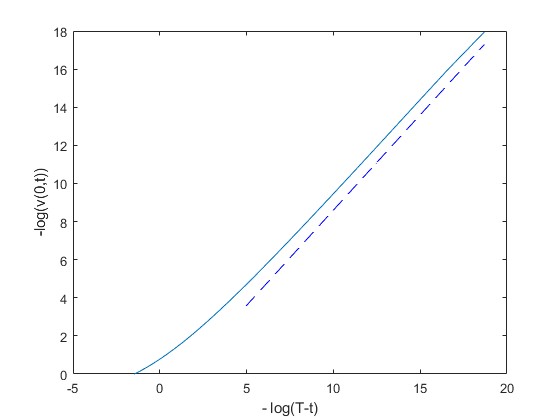}
\caption{$\lambda=0.1=\mu$, $p=2$ and $q=0.7$. To the left, the evolution of $(u(0,t),v(0,t))$. To the right, the quenching rate.}
\end{figure}

\begin{figure}[ht] \centering
\includegraphics[width=7.7cm]{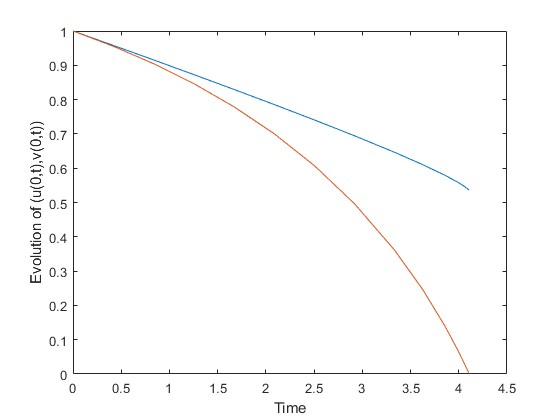}
\qquad
\includegraphics[width=7.7cm]{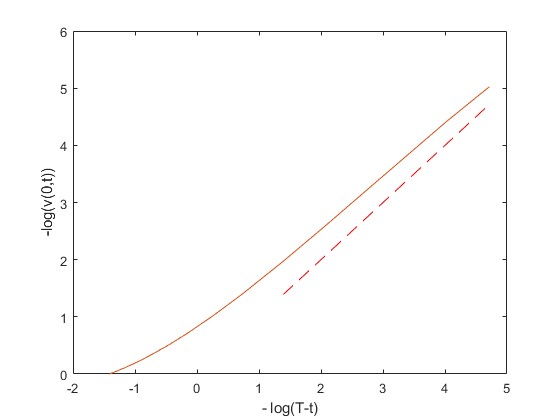}
\caption{$\lambda=0.1=\mu$, $p=0.2$ and $q=3$. To the left, the evolution of $(u(0,t),v(0,t))$. To the right, the quenching rate.}
\end{figure}

\end{document}